\documentclass[11pt, reqno]{amsart}
\usepackage[all]{xy}
\usepackage{stmaryrd}
\usepackage{hyperref}
\usepackage{enumerate}
\usepackage{amssymb,amsfonts}
\newcommand{\comment}[1]{}

\newtheorem{thm}{Theorem}[section]
\newtheorem{prop}[thm]{Proposition}

\newtheorem{cor}[thm]{Corollary}

\theoremstyle{definition}
\newtheorem{defn}[thm]{Definition}
\newtheorem{ex}[thm]{Example}

\theoremstyle{remark}
\newtheorem{rem}[thm]{Remark}

\newcommand{\Rr}{\mathbb R} 										
\newcommand{\Zz}{\mathbb Z}											
\newcommand{\set}[1]{\left\{#1\right\}}					

\newcommand{\tto}{\rightrightarrows}						

\newcommand{\F}{\ensuremath{\mathcal{F}}}

\newcommand{\NN}{\ensuremath{\mathcal{N}}}

\renewcommand{\graph}{\operatorname{Graph}}   

\DeclareMathOperator{\Diff}{Diff}               
\DeclareMathOperator{\Lie}{\mathcal{L}}         
\DeclareMathOperator{\Iso}{Iso}         
\DeclareMathOperator{\id}{id}    
\renewcommand{\d}{\mathrm d}                    
\newcommand{\Ss}{\mathbb S}											
\newcommand{\X}{\ensuremath{\mathfrak{X}}}			

\DeclareMathOperator{\Ver}{Ver}        
\renewcommand{\Vert}{{\operatorname{Ver}}}

\DeclareMathOperator{\Hor}{Hor}         
\DeclareMathOperator{\Gau}{Gau}       
\DeclareMathOperator{\Curv}{Curv}    

\DeclareMathOperator{\HOR}{H} 

\newcommand{\G}{\mathcal{G}}            
\newcommand{\M}{\mathcal{M}}            
\newcommand{\C}{\mathcal{C}}            
\newcommand{\s}{\mathbf{s}}             
\renewcommand{\t}{\mathbf{t}}           
\renewcommand{\F}{\mathcal{F}}          
\newcommand{\al}{\alpha}                
\newcommand{\be}{\beta}                 

\newcommand{\TM}{\mathbb{T}M}             

\renewcommand{\gg}{\mathfrak{g}}        
\newcommand{\g}{\mathfrak{g}}           
\newcommand{\z}{\mathfrak{z}}           

\DeclareMathOperator{\GL}{GL}            
\DeclareMathOperator{\gl}{\mathfrak{gl}} 

\DeclareMathOperator{\Ker}{Ker}         
\DeclareMathOperator{\im}{Im}           

\newcommand{\Ac}{\ensuremath{\mathbf{A}}} 
\newcommand{\ac}{\ensuremath{\mathbf{a}}} 
\newcommand{\Aci}{\ensuremath{\mathfrak{A}}} 
\newcommand{\aci}{\ensuremath{\mathfrak{a}}} 
\newcommand{\can}{\ensuremath{\mathrm{can}}} 

\DeclareMathOperator{\Ad}{Ad}           
\DeclareMathOperator{\Aut}{Aut}         
\DeclareMathOperator{\Der}{Der}         
\DeclareMathOperator{\InnAut}{InnAut}   
\DeclareMathOperator{\InnDer}{InnDer}   

\newcommand{\ri}{\overrightarrow}
\renewcommand{\le}{\overleftarrow}

\usepackage[dvips]{graphicx}

\begin{document}

\title{Integration of Coupling Dirac Structures}

\author{Olivier Brahic}
\address{Department of Mathematics, Federal University of Paran\'a.
CP 19081, 81531-980, Curitiba, PR, Brazil.}
\email{brahicolivier@gmail}

\author{Rui Loja Fernandes}
\address{Department of Mathematics,
University of Illinois at Urbana-Champaign,
1409 W.~Green Street,
Urbana, IL 61801
USA } 
\email{ruiloja@illinois.edu}

\thanks{Partially supported by NSF grant DMS 1308472, and by CNPq grant 401253/2012-0. Both authors acknowledge the support of the \emph{Ci\^encias Sem Fronteiras} program.}


\begin{abstract}
Coupling Dirac structures are Dirac structures defined on the total space of a fibration, generalizing hamiltonian fibrations from symplectic geometry,
where one replaces the symplectic structure on the fibers by a Poisson structure.  We study the associated Poisson gauge theory, in order to describe the 
presymplectic groupoid integrating coupling Dirac structures. We find the obstructions to integrability and we give explicit geometric descriptions of the integration. 
\end{abstract}

\maketitle

\section{Introduction}             %
\label{sec:introduction}           %

A Dirac structure on a manifold $M$ is a (possibly singular) foliation of $M$ by presymplectic leaves. It is well known that Dirac structures can be expressed in terms of a Lagrangian subbundle $L$ of the generalized tangent bundle $TM\oplus T^*M$. The bundle $L$ inherits a Lie algebroid structure from the Courant bracket \cite{Cou}, so Dirac structures are infinitesimal objects. In \cite{BCWZ} the authors showed that the global object underlying a given Dirac structure $L$ is a \emph{presymplectic groupoid}, i.e., a Lie groupoid $\G\tto M$ with a multiplicative closed 2-form $\Omega_\G$ satisfying a certain non-degeneracy condition. Not all Lie algebroids can be integrated to Lie groupoids, and Dirac structures are no exception: not all Dirac structures can be integrated to presymplectic groupoids. The obstructions to integrability follow from the general obstruction theory discovered in \cite{CrFe2,CrFe1}.

The general methods presented in \cite{CrFe2,CrFe1} allow one to decide if a given Lie algebroid is integrable or not, and to produce a canonical integration in terms of an abstract path space construction. While the obstructions to integrability can be computed explicitly in many examples, describing the canonical integration $\G(L)\tto M$ of a given integrable Dirac structure $(M,L)$ is, in general, a very difficult task. However, for a few classes of Dirac structures one does have explicit integrations and often in such cases the construction of the groupoid has a nice geometric flavor.

In this paper we discuss the integration of Dirac structures which arise as \emph{couplings}. The simplest example of a coupling Dirac structure is obtained in the context of \emph{symplectic fibrations}. If $p:E\to B$ is a symplectic fibration, a \textbf{coupling form} is a closed 2-form $\omega\in\Omega^2(E)$ on the total space of the fibration whose pullback to each fibre $F_b$ is the symplectic form $\omega_{F_b}$ on the fibre. The obstructions to the existence of such a coupling form are well known and we will recall them below. We are interested in the more general situation of a \emph{Poisson fibration}: now one looks for a coupling Dirac structure on the total space of the fibration which glues the Poisson structures on the fibers. This idea of a coupling is only a rough approximation: Dirac structures are very flexible and extra care must be taken in formulating precisely what a coupling Dirac structure is (\cite{BrFe,Vai}).

Coupling Dirac structures appear very natural in Poisson and Dirac geometry. One reason is that tubular neighborhoods of symplectic and presymplectic leaves in arbitrary Poisson and Dirac manifolds are always coupling Dirac structures. Our first main result concerning the integration of couplings can be stated as follows: 

\begin{thm}
\label{thm:1}
Let $L$ be a coupling Dirac structure on $p:E\to B$. If $L$ is integrable and $(\G,\Omega)\tto B$ is a source connected, presymplectic groupoid integrating $L$, then $\Omega$ is a coupling form for a fibration $\bar{p}:\G\to\Pi(B)$ obtained by integrating the algebroid morphism $p_*\circ\sharp:L\to TB$.
\end{thm}

In other words, coupling Dirac structures integrate to coupling forms. Moreover, one can express the geometric data of the integration in terms of the geometric data associated with the coupling Dirac structure $L$. As a consequence of this result, any presymplectic groupoid integrating a coupling Dirac structure on $p:E\to B$ is Morita equivalent to a symplectic groupoid integrating the induced vertical Poisson structure on a fiber $E_b$.

The previous result describes the symplectic geometry of the integration. One is also interested in the groupoid structure of the integration and the obstructions to integrability. Our inspiration to deal with this integration problem comes from a beautiful gauge construction, known as the \emph{Yang-Mills setup}, which yields coupling Dirac structures (\cite{BrFe2,GLS,GS,MaSa,Wein2,Wa}). One starts with a principal $G$-bundle $P\to B$, with a connection $\Gamma$, and a Hamiltonian action $G\times F\to F$ on a Poisson manifold $(F,\pi_F)$, and constructs a coupling Dirac structure $L$ on the associated bundle $E=P\times_G F$ extending the Poisson structures on the fibers. This construction can be further twisted by a closed 2-form on the base $B$ and it leads to many examples of coupling Dirac structures.

We show that one can integrate a Yang-Mills phase space as follows:
\begin{enumerate}[(i)]
\item Integrate first the fiber $(F,\pi_F)$ to a symplectic groupoid $\F\tto F$, and the principal $G$-bundle $P\to B$ to the gauge groupoid $\G(P)\tto B$;
\item Then integrate the vertical Poisson structure $\Ver^*$ to a \emph{fibered symplectic groupoid} $\G_V=P\times_G \F\tto E$;
\item The gauge groupoid $\G(P)\tto B$ acts on the fibered groupoid $\G_V\tto E\to B$, yielding a semi-direct product groupoid $\G(P)\ltimes \G_V\tto E$.
\item  Finally, the integration of the Yang-Mills phase space is a quotient 
\[ \G(L)=\G(P)\ltimes \G_V/\,\C, \] 
where $\C$ is a certain \emph{curvature groupoid}.
\end{enumerate}
Along the way we obtain the obstructions to integrability of a Yang-Mills phase space. Our integration procedure does not uses the principal bundle connection. Hence, all the different couplings obtained by varying the connection have the same integrating Lie groupoid $\G\tto E$. On the other hand, we also provide a construction for the presymplectic form $\Omega_\G$, which obviously depends on the choice of principal connection.

We show that if one is willing to accept infinite dimensional principal bundles, every coupling on a locally trivial fibration arises as a Yang-Mills phase space. This provides us with the clue to integrate arbitrary coupling Dirac structures as follows:

\begin{thm}
\label{thm:2}
Let $L$ be a coupling Dirac structure on $E\to B$. The source 1-connected groupoid $\G(L)$ integrating $L$ naturally identifies with equivalence classes in $P(TB)\ltimes_B \G(\Ver^*)$ under a certain equivalence relation:
 \begin{itemize}
 \item $(\gamma_0,g_0)\sim(\gamma_1,g_1)$ if and only if $\exists$ a homotopy $\gamma_B:I\times I\to B$, $(t,\epsilon)\mapsto \gamma^\epsilon_B(t)$ between $\gamma_0$ and $\gamma_1$, such that $g_1=\partial (\gamma_B,\t(g_0)).g_0.$
 \end{itemize}
where $\partial:P(TB)\times_B E\to \G(\Ver^*)$ is a certain ``groupoid" homomorphism, that can be computed explicitly.
\end{thm}

The quotes in ``groupoid" are here to remind that the path space $P(TB)$ is not really a groupoid, since associativity only holds up to isomorphism. 

Again, Theorem \ref{thm:2} should be viewed as an infinite dimension version of the groupoid integrating the Yang-Mills phase space. It also gives rise to the integrality obstructions of coupling Dirac structure. Namely, one checks that the restriction of the map $\partial:P(TB)\times_B E\to \G(\Ver^*)$ to a sphere in $B$ based at some $b\in B$ (seen as a map $\gamma_B:I^2\to B$ such that $\gamma_B(\partial I^2)=\{b\}$) is independent of its homotopy class. Then if we let $\M:=\partial (\pi_2(B)\times_B E)$, which we call the \emph{monodromy groupoid of the fibration}, we have:

\begin{thm} 
\label{thm:3}
Let $L$ be a coupling Dirac structure on $E\to B$ and assume that the associated connection $\Gamma$ is complete. Then $L$ is an integrable Lie algebroid if and only if the following conditions hold:
\begin{enumerate}[(i)]
\item the typical Poisson fiber $(E_x,\pi_V|_{E_x})$ is integrable,
\item the injection $\mathcal{M}\hookrightarrow \G(\Ver^*)$ is an embedding.
\end{enumerate}
\end{thm}

The transgression map $\partial:\pi_2(B)\times_B E\to \G(\Ver^*)$ is computable in many examples, and so are the integrability obstructions of Theorem \ref{thm:3}. We refer to the last section of the paper, where we will discuss for example the trivial fibration $p:\Ss^2\times\mathfrak{so}^*(3)\to \Ss^2$, with the usual Lie-Poisson structure on the fibers. Using Theorem \ref{thm:3} one can see that there is only a 2-parameter family of integrable Dirac couplings of rank 4, while there is an infinite dimensional family of non-integrable Dirac couplings of rank 4.
\newpage


\tableofcontents

\section{Coupling Dirac structures}%
\label{sec:Dirac:connections}      %

The notion of a \emph{coupling} was first introduced in the context of Dirac geometry in \cite{Vai} and \cite{BrFe}, but their origins lie in the theory of symplectic and hamiltonian fibrations (see, e.g., \cite{GLS}). In this section we will recall the definition of a \emph{coupling Dirac structure} and study its first properties.

\subsection{Dirac Structures} We first recall basic definitions concerning Dirac structures, mainly to fix notations and sign conventions. For details, see e.g.~\cite{Cou}.

Given a a smooth manifold $M$, we will denote by $\TM:=TM\oplus T^*\!M$ its generalized tangent bundle. The space of sections $\Gamma(\TM)=\X(M)\times\Omega^1(M)$ has natural pairings defined by:
\begin{equation}
 \label{eq:pairings}
\langle (X,\al),(Y,\be)\rangle_\pm:=\frac{1}{2}\left(i_Y\al\pm i_X\be\right),
\end{equation}
and a skew-symmetric bracket, called the \textbf{Courant bracket}, given by:
\begin{equation}
  \label{eq:Courant:bracket}
  \llbracket(X,\al),(Y,\be)\rrbracket:=
  \bigl([X,Y],\Lie_X\be-\Lie_Y\al+\d\langle (X,\al),(Y,\be)\rangle_-\bigr).
\end{equation}

\begin{defn} Given a smooth manifold $M$, an \textbf{almost Dirac structure} $L$ on $M$  is a subbundle $L\subset \TM:=TM\oplus T^*\!M$ of the generalized tangent bundle, which is maximal isotropic with respect to $\langle~,~\rangle_+$. An almost Dirac structure is said to be a \textbf{Dirac structure} if it is furthermore closed under the bracket $\llbracket~,~\rrbracket$.
\end{defn}

In general, the Courant bracket \emph{does not} satisfy the Jacobi identity.  For a Dirac structure $L$, however, the restriction of the bracket $\llbracket~,~\rrbracket$ to $\Gamma(L)$ yields a Lie bracket and if we let $\sharp:L\to TM$ be the restriction of the projection to $TM$, then $(L,\llbracket~,~\rrbracket,\sharp)$ defines a Lie algebroid. Each leaf of the corresponding characteristic foliation (obtained by integrating the singular distribution $\im\sharp$) carries a pre-symplectic form $\omega$: if $X,Y\in \im\sharp$, we can choose $\al,\be\in T^*\!M$ such that
$(X,\al),(Y,\be)\in L$ and set:
\begin{equation}
  \label{eq:pre:symp:form}
  \omega(X,Y):=\langle (X,\al),(Y,\be)\rangle_-=i_Y\al=-i_X\be.
\end{equation}
One can check that this definition is independent of choices and that $\omega$ is indeed closed. Thus we may think of a Dirac manifold as a (singular) foliated
manifold by pre-symplectic leaves.

\subsection{Fiber non-degenerate Dirac structures}

In what follows, unless otherwise stated, by a \emph{fibration} $p:E\to B$ we mean a surjective submersion.

\begin{defn}
Let $p:E\to B$ be a fibration. An almost Dirac structure $L$ on $E$ is called \textbf{fiber non-degenerate} if 
\begin{equation}
\label{eq:fiber:non:degnrt}
(\Vert\oplus\Vert^0)\cap L=\{0\}.
\end{equation}
Here $\Vert:=\ker\ p_*\subset TE$ denotes the vertical distribution, and $\Vert^0\subset T^*\!E$ its annihilator.
\end{defn}

In the terminology of \cite{FrMa}, when $L$ is a Poisson structure, this condition means that the fibers of $p:E\to B$ are \emph{Poisson transversals}. 

In order to understand the geometric meaning of this definition, one needs to decompose a fiber non-degenerate almost structure $L$ into its various components. First, $L$ gives rise to an \emph{Ehresmann connection}, by setting:
\begin{equation}
\label{eq:hor:space}
\Hor:=\set{X\in TE:\exists\, \al\in(\Vert)^0, (X,\al)\in L}.
\end{equation}
The fact that $\Hor\oplus \Vert=TE$ follows easily from \eqref{eq:fiber:non:degnrt}. 

It follows from \eqref{eq:hor:space} that the horizontal distribution $\Hor$ is contained in the characteristic distribution of $L$. Hence, we obtain a 2-form $\omega_H\in\Omega^2(\Hor)$ by restricting  the natural 2-form on the characteristic distribution to $\Hor$. More precisely, \eqref{eq:fiber:non:degnrt} and \eqref{eq:hor:space} together show that for each $X\in\Hor$, there
exists a unique $\al\in\Vert^0$ such that $(X,\al)\in L$. The skew-symmetric bilinear form $\omega_H:\Hor\times\Hor\to \Rr$ is obtained as:
\begin{equation}
\label{eq:2:form}
\omega_H(X_1,X_2):=\langle (X_1,\al_1),(X_2,\al_2)\rangle_-,
\end{equation}
where $\al_1,\al_2\in\Vert^0$ are the unique elements such that
$(X_1,\al_1),(X_2,\al_2)\in L$. Since $L$ is maximal isotropic, this $2$-form can also be written:
\begin{equation}
\label{eq:2:form:alt}
\omega_H(X_1,X_2)=\al_1(X_2)=-\al_2(X_1),
\end{equation}
so we obtain a smooth $2$-form $\omega_H\in\Omega^2(\Hor)$.  

Finally, we can associate to $L$ a vertical bivector field $\pi_V \in\X^2(\Vert)$. To see this, first observe that the annihilator of the horizontal distribution is:
\begin{equation}
\label{eq:annh:hor:space}
\Hor^0=\set{\al\in T^*\!E:\exists X\in \Vert, (X,\al)\in L}.
\end{equation}
This, together with \eqref{eq:fiber:non:degnrt}, shows that for each $\al\in\Hor^0$, there 
exists a unique $X\in\Vert$ such that $(X,\al)\in L$. Then one can
define a skew-symmetric bilinear form $\pi_V:\Hor^0\times\Hor^0\to \Rr$ by letting:
\begin{equation}
\label{eq:2:vector}
\pi_V(\al_1,\al_2):=\langle (X_1,\al_1),(X_2,\al_2)\rangle_-,
\end{equation}
where $X_1,X_2\in\Vert$ are the unique elements such that
$(X_1,\al_1),(X_2,\al_2)\in L$. Since $L$ is maximal isotropic the form $\pi_V:\Hor^0\times\Hor^0\to \Rr$ can also be written as:
\begin{equation}
\label{eq:2:vector:alt}
\pi_V(\al_1,\al_2)=\al_1(X_2)=-\al_2(X_1).
\end{equation}
Notice that the splitting $TE=\Hor\oplus\Vert$ allows us to
identify $\Hor^0=\Vert^*$, thus $\pi_V$ becomes a bivector field on the
fibers of $p:E\to B$. 

A more geometric interpretation of $\pi_V$ is that it is formed by the pullback to each fiber of the Dirac structure $L$: if one fixes a fiber $i_b:F_b=p^{-1}(b)\hookrightarrow E$, then $TF_b=\Vert|_{F_b}$ and one can identify $T^*\!F_b=\Vert^*|_{F_b}\simeq\Hor^0$ by using the connection. The pull-back Dirac structure $i^*L$ is then given by:
\begin{align*}
  i^*L&=\set{(X,\al|_\Vert)\in \Vert\oplus\Vert^*: (X,\al)\in L}\\
      &=\set{(X,\al)\in \Vert\oplus\Hor^0: (X,\al)\in L}\\
      &=\set{(X,\al)\in \Vert\oplus\Hor^0: X=\pi_V(\al,\cdot)}
      =\graph(\pi_V),
\end{align*}
where the third equality follows from \eqref{eq:2:vector:alt}. 

The preceding discussion justifies the following definition:

\begin{defn}
 A \textbf{geometric data} on a fibration $p:E\to B$ is a triple $(\pi_V,\Gamma,\omega_H)$ where
\begin{itemize}
\item $\pi_V\in\X^2(\Vert)$ is a vertical bivector field.
\item $\Gamma$ is an \emph{Ehresmann} connection, whose horizontal distribution will be denoted $\Hor$,
\item $\omega_H\in\Omega^2(\Hor)$ is an horizontal 2-form,
\end{itemize}
\end{defn}

Then we have: 

\begin{prop}\label{prop:1-1fibernondeg:couplingD}
Given a fibration $E\to B$, there is a $1$-$1$ correspondence between fiber non-degenerate almost Dirac structures and geometric data on the fibration.
\end{prop}

\begin{proof}
We have seen above how to associate to a fiber non-degenerate almost Dirac structure $L$ geometric data $(\pi_V,\Gamma,\omega_H)$. Conversely, given a geometric data $(\Gamma_L,\omega_H,\pi_V)$ on a fibration $p:E\to B$, we can define an almost Dirac structure $L$ by setting:
\begin{equation}
\label{eq:Dirac:components}
 L:=\set{(X+\pi_V^\sharp(\al),i_X\omega_H+\al): X\in\Hor, \al\in\Hor^0}.
\end{equation}
Notice that using the identifications $\Hor^0=\Ver^*,\Ver^0=\Hor^*$, we obtain simply $L=\graph{\omega_H}\oplus\graph{\pi_V}$,
which will prove to be a meaningful way of presenting $L$ later on.
\end{proof}

Given a fiber non-degenerate almost Dirac structure $L$, with associated geometric data $(\Gamma,\omega_H,\pi_V)$, we now express the conditions on this data which will guarantee that $L$ is a Dirac structure, i.e., that it is closed under the Courant bracket.

Let us first introduce some notations associated with the connection $\Gamma$. For a vector field $v\in\X(B)$ we denote by $h(v)\in\X(E)$ its horizontal lift. The exterior covariant differential $\d_\Gamma: \Omega^{k}(B)\otimes C^\infty(E)\to \Omega^{k+1}(B)\otimes C^\infty(E)$ is given by:
\begin{multline*} 
\d_\Gamma \omega(v_0,\dots,v_k)=\sum_{i=0}^ k (-1)^i \Lie_{h(v_i)} \omega(v_0,\dots,\hat{v}_i,\dots, v_k)+\\
+\sum_{i<j} (-1)^{i+j}\omega([v_i,v_j],v_0,\dots,\hat{v}_i,\dots,\hat{v}_j,\dots,v_k)
\end{multline*}
The curvature of $\Gamma$ will be denoted by $\Curv_\Gamma\in \Omega^2(B,\Ver)$ and is given by:
\[ \Curv_\Gamma(v,w)=[h(v),h(w)]-h([v,w])\quad (v,w\in\X(B)). \]
The curvature measures the failure of $\Hor$ in being involutive or, equivalently, the failure of $\d_\Gamma$ being a differential since we have:
\[ \d^2_\Gamma f(u,v)=\Lie_{\Curv(u,v)}f, \]
for any $f\in C^\infty(E)$ and $u,v\in\X(B)$. We can now state:

\begin{prop} 
\label{prop:obstr:Dirac}
  Let $(\pi_V,\Gamma,\omega_H)$ be the geometric data determined by a
  fiber non-degenerate almost Dirac structure $L$ on a fiber bundle
  $p:E\to B$. Then $L$ is a Dirac structure \emph{iff} the following conditions hold:
  \begin{enumerate}[i)]
    \item $\pi_V$ is a vertical Poisson structure: \[[\pi_V,\pi_V]=0.\]
    \item parallel transport along $\Gamma$ preserves the vertical Poisson structure: for any $v\in\X(B)$,
      \[ \Lie_{h(v)}\pi_V=0.\]
    \item the horizontal 2-form $\omega_H$ is $\Gamma$-closed: 
     \[ \d_\Gamma\omega_H=0.\] 
    \item the curvature is Hamiltonian: for any $u,v\in\X(B)$
      \begin{equation}
      \label{eq:curvature} 
      \Curv(u,v)=\pi_V^\#(\d i_{h(u)}i_{h(v)}\omega_H).
      \end{equation}
    \end{enumerate}
\end{prop}

A proof of Proposition \ref{prop:obstr:Dirac} can be found in \cite{BrFe}. We shall refer to \eqref{eq:curvature} as the \textbf{curvature identity}.
 
\begin{defn}
We call $L$ a \textbf{coupling Dirac structure} on the total space of a fibration $p:E\to B$ if $L$ is fiber non-degenerate and Dirac. Equivalently, if the associated geometric data $(\pi_V,\Gamma_L,\omega_H)$ satisfies all four conditions of Proposition \ref{prop:obstr:Dirac}. 
\end{defn}

\subsection{Examples}
\label{sub:sec:examples}

The notion of coupling Dirac structure contains as special cases the notion of \textbf{coupling form} for symplectic fibrations (see, e.g., \cite{GLS}) and the \textbf{coupling Poisson tensor} considered by Vorbjev in \cite{Vor}. We now recall these examples as well as other ones.


\subsubsection{Coupling forms}                                                                                                                          %
\label{sub:sec:examples:1}                                                                                                                            %

Let $\omega$ be a closed 2-form on the total space of a fibration $p:E\to B$. The associated Dirac structure $L:=\graph(\omega)$ is fiber non-degenerate if and only if the pull-back of $\omega$ to each fiber is a non-degenerate 2-form. In this case, the fibration with the restriction of $\omega$ to the fibers becomes a symplectic fibration. The geometric data $(\pi_V,\Gamma,\omega_H)$ associated to $L$ has a non-degenerate vertical Poisson structure $\pi_V$ which coincides with the inverse of the restriction of $\omega$ to the fibers. 

The converse is also true: a fiber non-degenerate Dirac structure $L$ for which the the geometric data $(\pi_V,\Gamma,\omega_H)$ has a non-degenerate vertical Poisson structure $\pi_V$ is determined by a presymplectic form $\omega$. In fact, it follows from \eqref{eq:Dirac:components} that $L$ is the graph of the closed 2-form:
\[ \omega=\omega_H\oplus(\pi_V)^{-1}.\]

Hence, fiber non-degenerate presymplectic forms are the same as coupling forms for symplectic fibrations \cite{GLS}.

\subsubsection{Coupling Poisson structures}											                %
\label{sub:sec:examples:2}                               												        %

Let $\pi$ be a Poisson structure on the total space of a fibration $p:E\to B$. The Dirac structure $L=\graph(\pi)$ is fiber non-degenerate if and only if $\pi$ is horizontal non-degenerate in the sense of Vorobjev \cite{Vor}, i.e., if the bilinear form $\pi|_{\Vert^0}:\Vert^0\times\Vert^0\to\Rr$ is non-degenerate. This is precisely the condition for the fibers to be \emph{Poisson transversals}, so there is an induced Poisson structure on the fibers for which the fibration becomes a Poisson fibration. In terms of the associated geometric data $(\pi_V,\Gamma_L,\omega_H)$ the Poisson structure on the fibers is $\pi_V$ and $\omega_H$ is non-degenerate: in fact, $\pi$ induces an isomorphism $\Vert^0\to\Hor$ and, using this  isomorphism, we see that $\omega_H$ coincides with the restriction $\pi|_{\Vert^0}$.

The converse is also true: a fiber non-degenerate Dirac structure $L$ for which the horizontal 2-form $\omega_H$  is non-degenerate, is given by a Poisson structure $\pi$. In fact, it follows from
\eqref{eq:Dirac:components} that: 
\[ \pi=(\omega_H)^{-1}\oplus \pi_V.\]
Hence, fiber non-degenerate Poisson structures are the same thing as
the horizontal non-degenerate Poisson structures of Vorobjev \cite{Vor}.

\subsubsection{Neighborhood of a pre-symplectic leaf}										        %
\label{sub:sec:examples:3}                               												%

Let $L$ be any Dirac structure on a manifold $M$ and fix a presymplectic leaf $(S,\omega)$ of $L$. Then the restriction of $L$ to any sufficiently small tubular neighborhood $p:\nu(S)\to S$ of the leaf is a coupling Dirac structure. To see this, one observes that along $S$:
\[ L_x\cap (\nu(S)_x\oplus \nu_x(S)^0)=\{0\}, \quad \forall x\in S. \] 
It follows that $L$ is fiber non-degenerate on a sufficiently small neighborhood of the zero section. 

This shows that in a neighborhood of a presymplectic leaf the Dirac structure takes a special form and we can associate to it geometric data  $(\pi_V,\Gamma,\omega_H)$. The Poisson structure $\pi_V$ is the transverse Poisson structure along $S$, while $S$ (viewed as the zero section) is an integral leaf of $\Hor$ and the 2-form $\omega_H$ restricted to this leaf coincides with $\omega$. Note however that, in general, the distribution $\Hor$ fails to be integrable at other points.

\comment{oli: added this example, because it is the historical one:}
\subsubsection{Reduction of canonical bundles}

Let $P\to M$ be a principal $G$-bundle. The action of $G$ naturally lifts to an hamiltonian action of $G$ on $(T^*\!P,\omega_\can)$. Clearly, $T^*\!P$ is itself a principal $G$-bundle, sometimes called a \textbf{canonical bundle}, and it follows that the base manifold $T^*\!P/G$ has an induced Poisson structure $\pi$. 

Each choice of a principal bundle connection $\theta:TP\to \g$ induces a projection map $p_\theta:T^*\!P/G\to T^*\!M$. It is easy to check that, for any choice of connection,  the Dirac structure $L=\graph(\pi)$ on $E=T^*\!P/G$ is a coupling Dirac structure over $B=T^*\!M$.

\subsubsection{Finite dimensional Yang-Mills-Higgs phase spaces}					    			                                 %
\label{sub:sec:examples:4}                               												%

There is a construction using principal bundles and hamiltonian actions which leads to an important class of coupling Dirac structures.

\begin{defn}
A classical Yang-Mills-Higgs setting is a triple $(P,G,F)$ where $P\to B$ is a principal $G$-bundle and $(F,\pi_F)$ is a $G$-hamiltonian Poisson manifold with equivariant moment map $\mu_F:F\to\g^*$. 
\end{defn}

\begin{prop}
Let $(P,G,F)$ be a classical Yang-Mills-Higgs setting. Each choice of a principal bundle connection $\theta:TP\to \g$ determines a coupling Dirac structure on the associated fiber bundle $E:=P\times_G F\to B$.
\end{prop}

The construction is well known (see \cite{BrFe2,Wein2,Wa}) so it will only be sketched here. 

First of all, the connection $\theta:TP\to \g$ gives a $G$-equivariant embedding $i_\theta:(\Ker \d p)^*\hookrightarrow T^*\! P$, where $p:P\to B$ is the principal bundle projection. This allows to pullback the hamiltonian $G$-space $(T^*\! P,\omega_\can,\mu_\can)$, where $\mu_\can:T^*\!P\to \gg^*$ is the dual of the infinitesimal action $\gg\to TP$, to obtain a Hamiltonian $G$-space $((\Ker \d p)^*,L_\theta,\mu_\theta))$. Here $L_\theta:=\graph(i_\theta^*\omega_\can)$ and $\mu_\theta:(\Ker \d p)^*\to\gg^*$ is the composition $\mu_\can\circ i_\theta$.

Next, we can combine the Hamiltonian $G$-spaces $((\Ker \d p)^*,L_\theta,\mu_\theta)$ and $(F,L_{\pi_F},\mu_F)$, where $L_{\pi_F}=\graph(\pi_F)$, to obtain a Hamiltonian $G$-space $((\Ker \d p)^*\times F,L_\theta\times L_{\pi_F},\mu_\theta+\mu_F)$, where $G$ acts diagonally on $(\Ker \d p)^*\times F$.

Finally, observe that the Hamiltonian quotient:
\[ ((\Ker \d p)^*\times F)/\!/G:=\{(v,f)\in (\Ker \d p)^*\times F: \mu_\theta(u)+\mu_F(f)=0\}/G \]
is isomorphic to $E:=P\times_G F$: the map $[(v,f)]\mapsto [(u,f)]$, where $v\in \Ker\d_u p$, gives the desired isomorphism. It follows that $E$ has a quotient Dirac structure $L$ and one checks that this is indeed a coupling Dirac structure for the fibration $E\to B$.

The associated coupling Dirac structure can be described as follows. Since $G$ acts on $F$ by Poisson automorphisms, the associated bundle $E:=P\times_G F$ has an induced vertical Poisson structure $\pi_V$ with typical fiber $(F,\pi_F)$. The induced connection $\Gamma$ on $E$ is a Poisson connection. If one denotes by  $\omega_\theta\in\Omega^2(B,\g)$ the curvature of the principal connection $\theta:TP\to\g$, one obtains a well defined horizontal $2$-form $\omega_H$ on $E$ by setting:
\[ \omega_H(h(v_1),h(v_2)):=\langle \mu_F,\omega_\theta(v_1,v_2)\rangle. \]
The triple $(\pi_V,\Gamma,\omega_H)$ is the geometric data associated with $L$. One can also check easily that this triple satisfies the conditions in proposition \eqref{prop:obstr:Dirac}.

Dirac structures obtained in this way are sometimes called classical \textbf{Yang-Mills-Higgs phase spaces}. We will be interested in the problem of integrability of such structures. In particular, the integrability of $(F,\pi_F)$ is not enough to ensure the integrability of the associated  bundle, as shown in the following simple example below.

\begin{ex}\label{ex:Hopf}
Consider the Hopf fibration $P=\mathbb{S}^3\to \mathbb{S}^2$, seen as a $\mathbb{S}^1$-principal bundle. One can choose a principal connection $\theta$ whose curvature is given by $\omega_\theta=p^*\omega$, where $\omega$ is the standard symplectic form on $\mathbb{S}^2$. Consider furthermore $F=\mathbb{R}$ endowed with the trivial Poisson structure $\pi_F=0$, and let $\mathbb{S}^1$ act trivially on $F$ with momentum map $f:F\to \mathbb{R}$ any smooth function.

Then the associated bundle is trivial $E=P\times_{\mathbb{S}^1} F=\mathbb{S}^2\times \mathbb{R}$, and it is easily checked that the induced coupling Dirac structure has presymplectic leaves $(\mathbb{S}^2\times\{x\}, f(x).\omega)$. Here, the associated gemetric data is given by $(\pi_V,\Hor,\omega_H)$, where $\pi_V=0$, $\Hor$ is the flat connection given by the trivialization,  and $\omega_H:=p^*\omega\cdot f$.

Although $\pi_V$ is integrable, it follows that the coupling Dirac structure $L$ is not integrable whenever $f$ has some critical point (see \cite{CrFe1}).
\end{ex}

\subsection{Poisson Gauge Theory}					    			                  %
\label{sub:sec:examples:5}                           	%
General coupling Dirac structures can be seen as infinite dimensional Yang-Mills-Higgs phase spaces. This was observed in \cite{BrFe} and often offers guidance on how to extend constructions which work for a Yang-Mills-Higgs phase space to a general coupling Dirac structure. We will use this principle later, so we will recall here this Poisson gauge theory. The discussion is mainly heuristic, since we do not want to deal with issues related with infinite dimensional Fr\'echet spaces.

Given a coupling Dirac structure $L$ on a fibre bundle $p:E\to B$ with associated geometric data $(\pi_V,\Gamma_L,\omega_H)$. We denote the fiber type by $(F,\pi_F)$ and we form the \textbf{Poisson frame bundle} over $B$, whose fiber over $b\in B$ is:
\[ P:=\set{ u:(F,\pi)\to (E_b,\pi_V): u\text{ is a Poisson diffeomorphism}}. \]
The group $G=\Diff(F,\pi_F)$ acts on (the right of) $P$ by pre-composition: 
\[ P\times G\to P:~(u,g)\mapsto u \circ g.\]
Then our original Poisson fiber bundle is canonically isomorphic to the associated fiber bundle: $E=P\times_G F$.

The Poisson connection $\Gamma_L$ determines a principal bundle connection on $P\to B$. To see this, observe that the tangent space $T_u P\subset C^\infty(u^*TE)$ at
a point $u\in P$ is formed by the vector fields along $u$, $X(x)\in T_{u(x)}E$, such that:
\[ d_{u(x)}p\cdot X(x)=\text{constant},\quad \Lie_X\pi_V=0.\]
The Lie algebra $\gg$ of $G$ is the space of Poisson vector fields: $\gg=\X(F,\pi)$. The infinitesimal action on $P$ is given by:
\[ \rho:\gg\to\X(P),\quad \rho(X)_u= \d u\cdot X,\]
so the vertical space of $P$ is:
\[ \Vert_u=\set{\d u\cdot X: X\in \X(F,\pi)}.\]
Now a Poisson connection $\Gamma_L$ on $p:E\to B$ determines a connection in $P\to B$ whose horizontal space is:
\[ \Hor_u=\set{\widetilde{v}\circ u: v\in T_bB},\]
where $u:F\to F_b$ and $\widetilde{v}:F_b\to T_{F_b }M$ denotes the horizontal lift of $v$.  Clearly, this defines a principal bundle
connection $\theta:TP\to\gg$ whose induced connection on the associated bundle $E=P\times_G F$ is the original Poisson connection $\Gamma_L$. 

Recall that the holonomy group $\Phi(b)$ with base point $b\in B$ is the group of holonomy transformations $\phi_\gamma:F_b\to F_b$,
where $\gamma$ is a loop based at $b$. Clearly, we have $\Phi(b)\subset\Diff(F_b,\pi_b)$. On the other hand, for $u\in P$ we
have the holonomy group $\Phi(u)\subset G=\Diff(F,\pi)$ of the corresponding connection in $P$ which induces $\Gamma_L$: it consist of
all elements $g\in G$ such that $u$ and $ug$ can be joined by a horizontal curve in $P$. These two groups are isomorphic, for if
$u:F\to F_b$ then:
\[ \Phi(u)\to \Phi(b),\ g\mapsto u\circ g\circ u^{-1},\]
is an isomorphism. 

The curvature of a principal bundle connection $\theta$ is a $\gg$-valued
2-form $\Theta$ on $P$ which transforms as:
\[ R_g^*\Theta=\Ad(g^{-1})\cdot \Theta,\quad (g\in G).\]
Therefore, we can also think of the curvature as a 2-form $\omega_\theta$
with values in the adjoint bundle $\gg_P:=P\times_G\gg$. In the case of the
Poisson frame bundle, the adjoint bundle has fiber over $b$ the space
$\X(F_b,\pi_b)$ of Poisson vector fields on the fiber. Hence the
curvature of our Poisson connection can be seen as a 2-form
$\omega_\theta:T_bB\times T_bB\to\X(F_b,\pi_b)$. The two curvature
connections are related by:
\begin{equation}
\label{eq:curv:relations}
\omega_\theta=\d u\circ \Theta\circ u^{-1}.
\end{equation}
Finally, using the curvature identity \eqref{eq:curvature}, we conclude that
\[ 
\Theta(v_1,v_2)_u=
\pi^{\#}\d (\omega_H(\widetilde{v}_1,\widetilde{v}_2)\circ u).
\]

If we fix $u_0\in P$, the Holonomy Theorem states that the Lie algebra
of the holonomy group $\Phi(u_0)$ is generated by all values
$\Theta(v_1,v_2)_u$, with $u\in P$ any point that can be connected
to $u_0$ by a horizontal curve. Hence, we can define a moment map
$\mu_F:F\to (\text{Lie}(\Phi(u_0))^*)$ for the action of $\Phi(u_0)$
on $F$ by:
\begin{equation}
\label{eq:moment:map}
\left\langle \mu_F(x),\Theta(v_1,v_2)_u\right\rangle
=\omega_H(\widetilde{v}_1,\widetilde{v}_2)_{u(x)}.
\end{equation}
This shows that the action of $\Phi(u_0)$ on $(F,\pi)$ is Hamiltonian.

We can now consider the principal bundle $\widetilde{P}\to B$ obtained from the Poisson frame bundle by reduction of the structure group $G$ to $\Phi(u_0)$. On
$\widetilde{P}\to B$ we have an induced connection $\omega_\theta$ and the action of $\Phi(u_0)$ on $(F,\pi)$ is hamiltonian with moment map $\mu_F:F\to (\text{Lie}(\Phi(u_0))^*)$. Then we obtain:

\begin{prop}
The Yang-Mills-Higgs phase space associated with the triple $(\widetilde{P},\Phi(u_0),F)$ is isomorphic to the original coupling Dirac structure $(E,L)$.
\end{prop}

\begin{proof}
By construction, the geometric data associated with the Yang-Mills-Higgs phase space of the triple $(\widetilde{P},\Phi(u_0),F)$ is precisely $(\pi_V,\Gamma_L,\omega_H)$.
\end{proof}

Hence \emph{any} coupling Dirac structure on a fibre bundle can be seen as a Yang-Mills-Higgs phase space, provided we allow for infinite dimensional structure groups.

\subsection{Coupling Dirac structures as extensions}

The Yang-Mills-Higgs approach to coupling Dirac structures has the disadvantage that one must allow for infinite dimensional structure groups. Since our purpose is to integrate coupling Dirac structures to presymplectic groupoids, this is problematic. An alternative approach is to observe that coupling Dirac structures give rise to Lie algebroid extensions:

\begin{prop}
\label{prop:Dirac:extension}
Let $L$ be a coupling Dirac structure on the fibration $p:E\to B$. The morphism $p_*\circ \sharp:L\to TB$ induces a Lie algebroid extension:                                                                  %
\begin{equation}
\label{eq:Dirac:extension}    
\SelectTips{cm}{}
\xymatrix@C=15pt{ \graph{\pi_V}\ \ar@{^{(}->}@<-0.25pt>[r]& L \ar@{->>}[r]<-0.25pt> &   TB      .\quad\quad\quad}
\end{equation} 
Moreover, the decompositon \eqref{eq:Dirac:components} induces an Ehresman connection with horizontal space $\graph(\omega_H)$, namely we have:            %
\begin{equation}\label{eq:Dirac:components2}  L=\graph(\pi_V)\oplus \graph(\omega_H).\end{equation}  %
\end{prop}

\begin{proof}
The map $p_*\!\circ\! \sharp:L\to TB$ is clearly a Lie algebroid morphism, being the composition of algebroid morphisms. Since $\Hor\subset\im\sharp$, it follows that $p_*$ is surjective and covers the surjective submersion $p:E\to B$. By its very definition (see \cite{Br}), we obtain a Lie algebroid extension with kernel $\ker (\d p\circ\sharp)$. The fiber non-degeneracy condition, shows that: 
\[ \ker p_*\circ \sharp=\sharp^{-1}(\Ver)=\graph(\pi_V). \]
Hence, the kernel is exactly $\graph(\pi_V)\subset \Ver^*\oplus\Ver$. The decomposition \eqref{eq:Dirac:components2} gives a complementary vector subundle to this kernel, i.e., an Ehresmann connection in the sense of \cite{Br}.
\end{proof}

For a Lie algebroid extension which is split, as in \eqref{eq:Dirac:components2}, there is a natural decomposition of its Lie bracket \cite[lem.\,1.8]{Br}.  

Firstly, we may identify $\graph(\pi_V)$ with $\Ver^*$, so sections of $\graph(\pi_V)$ get identified with vertical forms. Vertical forms  $\alpha,\,\beta\in \Gamma(\Ver^*)$ come naturally equipped with a bracket and an anchor inherited from $\pi_V$:
\begin{align}\sharp_V(\alpha)&:=\pi_V^\sharp(\alpha),\label{eq:vertical:anchor}\\
[\al,\be]_{V}&:=\Lie_{\sharp_V(\al)}\be-\Lie_{\sharp_V(\be)}\al-\d_V\pi_V(\al,\be),\label{eq:vertical:bracket}
\end{align}
where $\d_V:C^\infty(E)\to \Gamma(\Ver^*)$ denotes the vertical De Rham differential. Since $\pi_V$ is Poisson, this makes $\Ver^*$ into a Lie algebroid.

Secondly, we have the horizontal lift of  $v\in \X(B)$: it is defined as the unique section $h^*(v)$ of $\graph(\omega_H)$ such that $\d p\circ \sharp(h^*(v))=v$. Since the restriction of $\sharp$ gives an isomorphism $\sharp:\graph(\omega_V)\to\Hor$, we first have to lift $v$ into $h(v)\in \Gamma(\Hor)$ and then invert' $\sharp$, which leads to:
\[h^*(v)=(h(v),i_{h(v)}\omega_H)\in \Gamma(\graph\,\omega_H). \]
We shall refer to $h^*:\X(B)\to \Gamma(\graph\omega_H)$ as the \textbf{co-horizontal lifting map}. Note that this is a $C^\infty(B)$-linear map.

Sections of $L$ are generated by sections $\alpha\in\Gamma(\Ver^*)$ and $h^*(v)$, for $v\in\X(B)$, so the Lie bracket on $L$ is entirely determined by its value on these two types of sections:

\begin{prop}[Splitting Brackets]
\label{prop:splitbrackets} 
Let $L$ be a coupling Dirac structure on $E\to B$. Under the decomposition \eqref{eq:Dirac:components2} the Lie bracket of $L$ satisfies:
\begin{align}
    \left[\alpha,\beta\right]_L&= \left[\alpha,\beta\right]_{_V}                 \label{eq:splitbracket1}\\
    \left[h^*(v),\alpha\right]_L&=\mathcal{L}_{h(v)}\alpha                          \label{eq:splitbracket2} \\
    \left[h^*(v),h^*(w)\right]_L&=h^*([v,w])+\d_V\omega_H(h(v),h(w))      \label{eq:splitbracket3}
\end{align}
while the anchor takes the form:
\begin{equation}
\label{eq:splitbracket0}
\sharp(h^*(v)+\alpha)=h(v)+\sharp_V(\alpha) 
\end{equation}
for any elements $v,w\in\X(B)$ and $\alpha, \beta\in \Gamma(\Ver^*)$.
\end{prop}

\begin{proof}
Straightforward computation using \eqref{eq:Courant:bracket} and the identifications $\Hor^*\simeq \Ver^0$, $\Ver^*\simeq\Hor^0$.
\end{proof}

In particular, we see from the equation \eqref{eq:splitbracket3} that the curvature of $\graph(\omega_H)$ as an Ehresmann connection on the extension $L$ is given by:
\[ ( \pi_V^\sharp\d_V\omega_H, \d_V\omega_H)\in \Omega^2(B,\Gamma(\graph(\pi_V))). \]
This is just another way of expressing the curvature identity \eqref{eq:curvature}.

\section{Integration of coupling Dirac structures I}%

As stated in the Introduction, our main aim in this paper is to understand the integration of coupling Dirac structures. 
In this section we take care of the symplectic geometry, showing that a source connected integration of any integrable coupling Dirac structure is a presymplectic groupoid whose 2-form is itself a coupling form. 

\subsection{The Weinstein groupoid and $A$-homotopies}                                                                      

We start by recalling some facts about the integrability of Lie algebroids that will be needed later on. In the sequel, we will denote by $\G\tto M$ a Lie groupoid, with source and target maps $\s,\t:\G\to M$, identity section $\iota:M\to \G$, $m\mapsto {\bf 1}_m$,
and inversion $i:\G\to\G$, $x\mapsto x^{{\bf -1}}$. The composition of two arrows, denoted by
$x\cdot y$, is only defined provided $\s(x)=\t(y)$.

We will denote by $p_A:A\to M$ a Lie algebroid with Lie bracket $[~,~]_A$ 
and anchor $\sharp:A\to TM$. Given a Lie groupoid $\G$, the corresponding Lie algebroid has underlying vector bundle
$A(\G):=\ker\d_{\iota(M)}\s$ and anchor $\sharp:=\d_{\iota(M)}\t$. The sections of 
$A(\G)$ can be identified with the right invariant vector fields on $\G$, and this 
determines the Lie bracket on sections of $A(\G)$. A groupoid that arises in this way is called \emph{integrable}.

Not every Lie algebroid $p_{\!\,A}:A\to M$ is integrable. However, there always exists  a topological groupoid 
$\G(A)$ with source 1-connected fibers, that formally 'integrates' $A$. It is called the \textbf{Weinstein
groupoid} of $A$. Then $A$ is integrable if and only if $\G(A)$ is smooth, in which case $A(\G(A))$ is canonically isomorphic to $A$.

Let us recall briefly the construction  of $\G(A)$. More details can be found in \cite{CrFe3}. An \textbf{$A$-path} is a path $a:I\to A$ such that:
\[ \#a(t)=\frac{\d}{\d t}p_A(a(t)). \]
We will denote by $P(A)$ the space of $A$-paths (up to reparameterization). We set $\s(a):=p_{\!A}\circ\, a(0)$ and $\t(a):=p_{\! A}\circ\, a(1)$, and we will think of $P(A)\tto M$ as an infinite dimensional groupoid with multiplication given by concatenation (although units and inverse are not well defined). On the space $P(A)$, there is an equivalence relation $\sim$, called $A$-homotopy, that preserves the multiplication. The quotient of $P(A)$ by $A$-homotopies, usually called the Weinstein groupoid, will be denoted by:
\[ \G(A):=P(A)/\!\sim . \] 
Given an $A$-path $a$, we usually denote its $A$-homotopy class by $[a]_A$, or simply $[a]$ when no confusion seems possible.

Let us recall how $A$-homotopies are defined. Suppose that we 
are given $\alpha^\epsilon$(t), a time dependent family of sections of $A$ depending on a parameter $\epsilon\in I:=[0,1]$, and $\beta^0(\epsilon)$ a time dependent section of $A$.  
Then there exists a unique solution $\beta=\beta^t(\epsilon)$ of the following \textbf{evolution equation}:
\begin{equation}
\label{evolution}
\frac{\d\alpha}{\d\epsilon}-\frac{\d\beta}{\d t}=[\alpha,\beta],
\end{equation}
with initial condition $\beta^0(\epsilon)$. In fact, it is easily checked that the following integral formula provides a solution:
\begin{equation}
\label{hintegral}\beta^t(\epsilon):=\int_0^t(\psi^{\alpha^\epsilon}_{t,s})_*\bigl(\,\frac{\d}{\d\epsilon}\alpha^\epsilon(s)\,\bigr)\d s+(\psi_{t,0}^{\alpha^\epsilon})_*(\beta^0(\epsilon)).
\end{equation}
Here $\psi^{\alpha^\epsilon}_{t,s}$ denotes the flow of the time dependent linear vector field on $A$ associated with the derivation $[\alpha^\epsilon,-]_A$ (see \cite{CrFe3} for more details). We emphasize the use of the indices and parameters in the notation: we think of $\alpha$ as an $\epsilon$-family of $t$-time dependent sections of $A$, while we think of $\beta$ as a $t$-family of $\epsilon$-time dependent sections of $A$ (see why below).

The notion of $A$-homotopy is defined as follows: a family $a^\epsilon:I\mapsto A,\ \epsilon\in I$ of $A$-paths, over $\gamma^\epsilon:I\mapsto B$ is called an \textbf{$A$-homotopy} 
if and only the unique solution $\beta$ of equation \eqref{evolution} with initial condition $\beta^0(\epsilon)=0$ satisfies: 
\begin{equation}\label{hcondition}
\beta^1(\epsilon)_{\gamma^\epsilon(1)}=0, \quad (\epsilon\in I).
\end{equation}
Here, $\alpha^\epsilon$ denotes any family of time-dependant sections of $A$ {\bf extending} $a$, that is, such that $\alpha^\epsilon(\gamma^\epsilon(t),t)=a^\epsilon(t)$. One checks
that this definition is independent of the choice of $\alpha$ (see \cite{CrFe2}). We will refer to \eqref{hcondition} as the {\bf homotopy condition}. Note that, if some $\alpha^0$ is fixed, and we are given $\beta$ with $\beta^0=\beta^1=0$, then equation \eqref{evolution} can also be considered as an evolution equation for $\alpha$, which turns out to induce an homotopy.

Sometimes, it is useful to consider evolution equations \eqref{evolution} with non-vanishing initial conditions as well, because they also lead to $A$-homotopies, as we explain now. Let us set $X:=\sharp\alpha$ and $Y:=\sharp\beta$, so that $X^\epsilon$ and $Y^t$ are families of time dependent vector fields on $M$. Obviously, $X$ and $Y$ satisfy an evolution equation such as \eqref{evolution} in the algebroid $TM$. As the the proof of the next proposition will show, this forces their time dependent flows $\phi^{X^\epsilon}_{t,0}, \phi^{Y^t}_{\epsilon,0}$ to be related as follows:
\begin{equation}\label{hflows}
\phi^{X^\epsilon}_{t,0}\circ\phi^{Y^0}_{\epsilon,0}=\phi^{Y^t}_{\epsilon,0}\circ\phi^{X^0}_{t,0}. 
\end{equation}
In particular, if we denote by $\gamma^\epsilon(t)$ any of these two equivalent expressions applied to some $m_0\in M$, we obtain two families of $A$-paths $a^\epsilon$ and $b^t$, defined by $a^\epsilon(t):=\alpha^\epsilon(t)_{\gamma^\epsilon(t)}$ and $b^t:\epsilon\mapsto \beta^t(\epsilon)_{\gamma^\epsilon(t)}$. We have:

\begin{prop}
\label{hgeom}
For any pair of time dependent sections $\alpha$ and $\beta$ satisfying the evolution equation \eqref{evolution} and any $m_0\in M$, the concatenations $a^1\cdot b^0$ and $b^1\cdot a^0$ of the $A$-paths defined above are $A$-homotopic.
\end{prop}

\begin{proof}
We will assume that $A$ is integrable. In the case where $A$ is not integrable, one easily checks that the formulas given below still satisfy the necessary conditions to be an $A$-homotopy.

Let us denote by $\ri{\alpha}^\epsilon$ and $\ri{\beta}^t$ the (time dependent) right invariant vector fields on $\G(A)$ induced by $\alpha^\epsilon$ and $\beta^t$ respectively. Then the evolution equation \eqref{evolution} exactly means that $\ri{\alpha}+\partial_t$ and $\ri{\beta}+\partial_\epsilon$ commute
 as vector fields on $\G(A)\times I\times I$. Therefore, their respective flows:
\[
\begin{array}{cccc}
 \psi_u^\alpha(g,t,\epsilon)&=&(\phi^{{\alpha}^\epsilon}_{t+u,t}(g),t+u,\epsilon)\\
 \psi_v^\beta(g,t,\epsilon)&=&(\phi^{{\beta}^t}_{\epsilon+v,\epsilon}(g),t,\epsilon+v),
\end{array}
\]
commute. Writing dowm explicitly the commutation relation, we easily obtain that the corresponding time dependent flows satisfy:
\[
\phi^{{\alpha}^{\epsilon+v}}_{t+u,t}\circ\phi^{{\beta}^t}_{\epsilon+v,\epsilon}=\phi^{{\beta}^{t+u}}_{\epsilon+v,\epsilon}\circ \phi^{{\alpha}^{\epsilon}}_{t+u,t}.
\]
In particular, taking $t=\epsilon=0$, and then switching the roles of $u,v$ with the ones of $t,\epsilon$, we obtain the following relation: 
\[
\phi^{{\alpha}^{\epsilon}}_{t,0}\circ\phi^{{\beta}^0}_{\epsilon,0}({\bf 1}_{m_0})=
\phi^{{\beta}^{t}}_{\epsilon,0}\circ \phi^{{\alpha}^{0}}_{t,0}({\bf 1}_{m_0}).
\]
Now, observe that the path in $\G(A)$ corresponding to the concatenation $a^1\cdot b^0$ is 
the concatenation of the following paths in the source-fibers:
\[ 
\epsilon\mapsto  \phi^{{\beta}^0}_{\epsilon,0}({\bf 1}_{m_0}),\quad t\mapsto\phi^{{\alpha}^{1}}_{t,0}\circ\phi^{{\beta}^0}_{1,0}({\bf 1}_{m_0}),
\]
while the one corresponding to $b^1\cdot a^0$ is the concatenation of the paths:
\[
t\mapsto\phi^{{\alpha}^{0}}_{t,0}({\bf 1}_{m_0}),\quad \epsilon\mapsto\phi^{{\beta}^{1}}_{\epsilon,0}\circ \phi^{{\alpha}^{0}}_{1,0}({\bf 1}_{m_0}).
\]
These are clearly homotopic (in the $\s$-fibers) since they define the boundary of a square given by:
\[
(t,\epsilon)\mapsto \phi^{{\alpha}^{\epsilon}}_{t,0}\circ\phi^{{\beta}^0}_{\epsilon,0}({\bf 1}_{m_0})=\phi^{{\beta}^{t}}_{\epsilon,0}\circ \phi^{{\alpha}^{0}}_{t,0}({\bf 1}_{m_0}).
\]
We conclude that $a^1\cdot b^0$ and $b^1\cdot a^0$ are homotopic as $A$-paths.
\end{proof}

\begin{rem}
When $b^0=0$, we can also interpret Proposition \ref{hgeom} in terms of $A$-homotopies: 
it says that $b^1$ is a representative (up to $A$-homotopy) of $a^1\cdot (a^0)^{-1}$. 
Hence, $b^1$ is trivial if and only if $a^0$ is homotopic to $a^1$.
\end{rem}

\subsection{Presymplectic groupoids}
\label{sub:symp:grpds}
When a Lie algebroid arises from a geometric structure, the corresponding Lie groupoid usually inherits an extra geometric structure. We are interested in the case of Dirac structures $L$ and in this case the Weinstein groupoid $\G(L)$ comes equipped with a multiplicative presymplectic form, as we briefly recall here (see \cite{PoWa} or \cite{BCWZ} for more details).

\begin{defn} 
\label{defn:multiplicative:form}
A  $2$-form $\Omega\in\Omega^2(\G)$ defined on the space of arrows of a Lie groupoid is said to be multiplicative if the following relation holds:
 \[m^*\Omega=pr_1^*\Omega+pr_2^*\Omega,\]
where $m:\G^{(2)}\to\G$ is the multiplication of composable arrows and $pr_i:\G^{(2)}\to\G$ the projection on both factors. Ones calls a \textbf{presymplectic groupoid} a Lie groupoid endowed with a multiplicative $2$-form $\Omega$ such that:
\begin{equation}
\label{eq:non-degeneracy}
\ker \Omega_x\cap \ker(\d \s)_x\cap \ker(\d \t)_x=\{0\},\quad \forall x\in M.
\end{equation}
\end{defn}

Roughly speaking, the following result says that Dirac structures integrate to presymplectic groupoids.
\begin{prop}
 Let $L$ be a Dirac structure on a manifold $M$. If $L$ is integrable, then  $\G(L)$  has a naturally induced multiplicative presymplectic form such that the map $(\t,\s):(\G(L),\Omega)\to (M\times M, L\times L^\text{opp})$ is $f$-Dirac.
\end{prop}

The aforementioned multiplicative presymplectic form $\Omega$ on $\G(L)$ is related to sections of $L$ in the following way: for any any $X\in T\G$ and any pair of sections $\eta=(v,\alpha),\xi=(w,\beta)\in\Gamma(L)$, one has:
\begin{align*}
\Omega(\le{\eta},X)&=-\alpha(\s_*X),\\
\Omega(\ri{\xi},X)&=\beta(\t_*X),
\end{align*}
where we denoted by $\le{\eta}$ (resp. $\ri{\xi}$) the left (resp. right) invariant vector field on $\G(L)$ associated to $\eta$ (resp. $\xi$). Also, source and target fibers turn out to be presymplectically orthogonal: 
\[ \Omega(\le{\eta},\ri{\xi})=0. \]
Finally, if $(S,\omega_S)$ is the presymplectic leaf of $(M,L)$ through $x\in M$, then $\t:\s^{-1}(x)\to S$ is a principal $G_x$-bundle and one has:
\[ i_{s^{-1}(x)}^*\Omega=\t|_{s^{-1}(x)}^*\omega_S, \]
where $i_{s^{-1}(x)}:s^{-1}(x)\hookrightarrow \G(L)$ denotes the inclusion.

\begin{ex}
The Dirac structure $L=D\oplus D^0$ associated with a integrable distribution $D=T\F$ is always integrable:  $\G(L)=\Pi_1(\F)\ltimes\nu^*(\F)$ is the groupoid obtained from the linear holonomy action of the homotopy classes of paths in $\F$ on the conormal bundle. The presymplectic form on $\G(L)$ is $\Omega=\pi_2^*\omega$ where $\pi_2: \Pi_1(\F)\ltimes\nu^*(\F)\to \nu^*(\F)$ is the projection in the second factor, and $\omega$ is the pullback under the inclusion $\nu^*(\F)\hookrightarrow T^*M$ of the canonical symplectic form on the cotangent bundle.
\end{ex}

\begin{ex}
The Dirac structure $L=\graph(\omega)$ associated with a closed 2-form $\omega\in\Omega^2(M)$ is always integrable: $\G(L)=\Pi_1(M)=\widetilde{M}\times_{\pi_1(M)} \widetilde{M}$, where $\widetilde{M}$ is the universal covering space of $M$. The presymplectic form on $\G(L)$ is $\Omega=\t^*\omega-\s^*\omega$.
\end{ex}

\begin{ex}
The Dirac structure $L=\graph(\pi)$ associated with a Poisson structure may fail to be integrable (see \cite{CrFe1} for a geometric description of the obstructions). If $L$ happens to be integrable, then the multiplicative 2-form $\Omega$ on $\G(L)$ is actually a \emph{symplectic} form. 
\end{ex}

Note that given an integrable Dirac structure $(M,L)$ there can be other presymplectic groupoids $(\G,\Omega_\G)$ integrating $(M,L)$ besides $(\G(L),\Omega)$. 
However, if $(\G,\Omega_\G)$ has source connected fibers, then there is a covering Lie groupoid homomorphism $\Phi:(\G(L),\Omega)\to (\G,\Omega_\G)$ with $\Phi^*\Omega_\G=\Omega$.

\subsection{Couplings integrate to couplings}                                                                               
\label{int:coupling}  			                         			                                                    

\comment{oli: I just added the completeness discussion, remained vague about $\tilde{p}$ being a locally trivial submersion, see if that is ok and erase this comment}

Assume now that $L$ is an integrable coupling Dirac structure on a fibration $p:E\to B$. The anchor $\sharp:L\to TE$ is a Lie algebroid morphism that integrates to the groupoid morphism $\G(L)\to \Pi(E)$ which associates to the homotopy class of an $L$-path the homotopy class of its base path. On the other hand, $p_*:TE\to TB$ is a Lie algebroid morphism whose integration $\Pi(E)\to\Pi(B)$ is the morphism $[\gamma]\mapsto [p\circ \gamma]$. We will denote the composition of this two groupoid morphisms by $\tilde{p}:\G(L)\to \Pi(B)$. 

The morphism $\tilde{p}:\G(L)\to \Pi(B)$ integrates the Lie algebroid morphism $p_*\circ\sharp:L\to TB$, which is surjective on the fibers, by the coupling condition. Hence, $\tilde{p}:\G(L)\to \Pi(B)$ is a submersion, which however does not need to be surjective. 

\begin{prop}
If $L$ is an integrable coupling Dirac structure on a fibration $p:E\to B$ and the induced Ehresmann connection is complete, then $\tilde{p}:\G(L)\to \Pi(B)$ is surjective and so it is a fibration.
\end{prop}

\begin{proof}
Given $[\gamma]\in\Pi(B)$, where $\gamma:I\to B$ is a smooth path, completeness allows us to lift $\gamma$ to an horizontal path $\widetilde{\gamma}:I\to E$. Since $\widetilde{\gamma}'(t)\in\im\sharp$, we can find an $L$-path $a:I\to L$ with base path $\widetilde{\gamma}$. Then $\tilde{p}([a])=[\gamma]$.
\end{proof}

\begin{rem}
One can show that if a locally trivial fibration $p:E\to M$ admits a complete Ehresmann connection, then $p_*:\Pi(E)\to\Pi(B)$ is also locally trivial and carries an induced Ehresmann connection. It follows then that if $L$ is an integrable coupling Dirac structure on a fibration $p:E\to B$ and the induced Ehresmann connection is complete, then $\tilde{p}:\G(L)\to \Pi(B)$ is also a locally trivial fibration.
\end{rem}

From now one we will make the implicit assumption that our coupling Dirac structures have complete induced connections. This is the case, e.g., if the fibers are compact.

\begin{thm}
\label{thm:grpdcoupling}
Let $L$ be a coupling Dirac structure on $p:E\to B$. If $L$ is integrable, then the multiplicative presymplectic form $\Omega$ on $\G(L)$ is fiber non-degenerate for the fibration 
\begin{equation} 
\label{int:fibration}
\tilde{p}:\G(L)\to \Pi(B),
\end{equation} 
obtained by integrating the Lie algebroid morphism $p_*\circ \sharp:L\to TB$. 
\end{thm}

\begin{proof}
Let us denote by $\Vert_{\G_L}:=\ker \tilde{p}$. We only need to check that the non-degeneracy condition \eqref{eq:fiber:non:degnrt} holds: 
\[ (\Ver_{\G(L)}\oplus\Vert_{\G(L)}^0)\cap \graph{\Omega}=\{0\}. \]

First notice that since $\tilde{p}$ is obtained by composing the groupoid maps $\G(L)\to \Pi(E)\to \Pi(B)$, it follows that $(X,\alpha)\in T\G(L)\oplus T^*\G(L)$ lies in $\Ver_{\G(L)}\oplus\Vert_{\G(L)}^0$ if and only if it satisfies the following two conditions:
\[ (\s_*\times \t_*)(X)\in\Ver\times \Ver, \quad \alpha \in  (\s^*\times \t^*)(\Ver^0\times \Ver^0). \]
Let $g\in \G(L)$ be the base point of $(X,\alpha)$ and set $x:=\s(g)\in E$ and $y:=\t(g)\in E$. The second condition shows that 
$\alpha\in \s^*(\Ver^0_x)+\t^*(\Ver^0_x)$, so there exists $a_0\in \Ver^0_x$ and $a_1\in\Ver^ 0_y$ such that $\alpha=\s^*a_0-\t^*a_1$. It follows from the first condition that:
\[ (\s_*X,a_0)\in \Ver_x\oplus \Ver^0_x,\quad (\t_*X,a_1)\in \Ver_y\oplus \Ver^0_y. \] 

Therefore, for any $(X,\alpha)\in (\Ver_{\G(L)}\oplus\Vert_{\G(L)}^0)\cap \graph{\Omega}$ we must have $(\s_*X,a_0)\in L_x$
and $(\t_*X,a_1)\in L_y$, since $\s \times \t$ is a forward Dirac Dirac map. By the fiber non-deneracy condition of $L$, we conclude that $(\s_*X,a_0)=0$ and $(\t_*X,a_1)=0$.

It follows that $\alpha=\s^*a_0-\t^*a_1=0$ and that $X \in \ker \s_*\cap \ker \t_*$. Since $(X,\alpha)\in \graph{\Omega}$ we conclude that $X\in \ker \s_*\cap \ker \t_*\cap \ker \Omega$. The non-degeneracy condition of $\Omega$ (see Definition \ref{defn:multiplicative:form}) shows that we must also have $X=0$.
\end{proof}

\begin{rem}
\label{rem:integr:fibers}
For each $b\in B$ the fiber $\tilde{p}^{-1}({\bf 1}_b)$ is a Lie subgroupoid of $\G(L)$ over the fiber $E_b:=p^{-1}(b)$ and the 
restriction of $\Omega$ to the fiber is symplectic: it is a symplectic groupoid integrating the vertical Poisson structure $(E_b,\pi_b)$ (the fact that $\ker p_*\circ\sharp=\sharp^{-1}(\Ver)$ identifies with $\graph{\pi_V}=\Ver^*$ as a Lie algebroid, is a consequence of  Proposition \ref{prop:splitbrackets}). The kernel of $\tilde{p}$ is also a Lie subgroupoid of $\G(L)$ over $E$ of a special kind, called a \emph{fibered symplectic groupoid}, which we will study in Section \ref{sec:fibered:sympl:grpds}.
\end{rem}

If $(\G,\Omega_\G)$ is another presymplectic groupoid integrating $(E,L)$ with source connected fibers, then we claim that there is also a fibration $\bar{p}:\G\to \G_B$, where $\G_B$ is a certain Lie groupoid integrating $TB$: in fact, since  $\G$ has source connected fibers, there is a covering homomorphism $\Phi:\G(L)\to \G$, whose kernel $\NN\subset \G(L)$ is an embedded bundle of normal subgroups. Its image $\tilde{p}(\NN)\subset\pi_1(B)$ is also an embedded bundle of normal subgroups and so the quotient $\G_B:=\pi_1(B)/\tilde{p}(\NN)$ is another Lie groupoid integrating $TB$. Moreover, we obtain a groupoid morphism $\bar{p}:\G\to \G_B$ from $\tilde{p}:\G(L)\to \pi(B)$ by passing to the quotient. We then obtain as a corollary of Theorem \ref{thm:grpdcoupling}:

\begin{cor}
\label{cor:coupling}
Let $L$ be a coupling Dirac structure on $E\to B$. If $(\G,\Omega_\G)$ is any source connected presymplectic groupoid integrating $L$ then $\Omega_G$ is a coupling form relative to the fibration $\bar{p}:\G\to \G_B$, the unique Lie groupoid homomorphism integrating the Lie algebroid morphism $\sharp\circ p_*$. 
\end{cor}

\subsection{Integration of the geometric data}                                                                              
\label{int:data}  			                         			                                                    

Let $L$ be an integrable coupling Dirac structure on $p:E\to B$, with associated geometric data $(\pi_V,\Gamma,\omega_H)$, and $(\G,\Omega)$ a source connected presymplectic groupoid integrating $L$. According to the results of the previous section, $\Omega$ is a coupling form relative to a fibration $\bar{p}:\G\to \G_B$, which is a Lie groupoid homomorphism integrating the the Lie algebroid morphism $p_*\circ\sharp$. We denote by $(\Omega_V,\tilde{\Gamma},\Omega_H)$ the corresponding geometric data. 

Our next result shows that one can obtain the geometric data of the coupling multiplicative 2-form $\Omega$ in terms of the geometric data of the coupling Dirac structure $L$:


\begin{prop}[Integration of the Geometric Data]\label{prop:int:geom:data}
The geometric data $(\Omega_V,\tilde{\Gamma},\Omega_H)$ for $\Omega$ is related to the geometric data $(\pi_V,\Gamma,\omega_H)$ for $L$ in the following way:
\begin{enumerate}[(i)]
\item $(\G_{E_b},\Omega_{E_b}):=(\tilde{p}^{-1}({\bf 1}_b),i_b^*\Omega_V)$ is a symplectic Lie groupoid over $E_b$, which integrates $\pi_V|_{E_b}$, where $i_b:\tilde{p}^{-1}({\bf 1}_b)\hookrightarrow \G$ is the inclusion.
\item The connection $\tilde{\Gamma}$ has horizontal lift given by:
  \begin{equation}\label{eq:HOR}\HOR(v,w)=\le{h}^{\!*}(v)-\ri{h}^{\!*}(w),\end{equation}
  where $h^*$ denotes the co-horizontal lift.
\item the horizontal form $\Omega_H$ is given by:
\begin{align}\Omega_H(H(v_1,w_1),H(v_2,w_2))&=\omega_H(h(v_1),h(v_2))\circ\t \nonumber\\&\label{eq:OMEGAH} \ \quad-\omega_H(h(w_1),h(w_2))\circ\s.\end{align}
\end{enumerate}
where we used the natural identification $T_g\G_B=T_{\t(g)}B\times T_{\s(g)}B$.
\end{prop}

\begin{proof}
Item (i) was already discussed in Section \ref{int:coupling} (see Remark \ref{rem:integr:fibers}). 

To prove item (ii), consider an element $(v,w)\in T_g\G_B=T_{\t(g)}B\times T_{\s(g)}B$. Using the expression $h^*(v):=(h(v),i_{h(v)} \omega_H)$ for the co-horizintal lifts, one checks that the right hand term in \eqref{eq:HOR} projects onto $(v,w)$ and lies in $\Ver_G^{\perp\Omega_L}$. By uniqueness, it must coincide with $\HOR(v,w)$.

Finally, expression \eqref{eq:OMEGAH} for the horizontal  2-form follows by straightforward computation, using the general properties of multiplicative $2$-forms:
\begin{align*}
 \Omega_H(\HOR(v_1,w_1),\HOR(v_2,w_2))
&=\Omega\bigl(\le{h^*}(v_1),\le{h^*}(v_2)\bigr)+\Omega\bigl(\ri{h^*}(w_1),\ri{h^*}(w_1)\bigr)\\
&=\bigl\langle(h(v_1),\eta_{v_1}),(h(v_2),\eta_{v_2})\bigr\rangle_-\circ \s \\
&\quad-\bigl\langle(h(w_1),\eta_{w_1}),(h(w_2),\eta_{w_2})\bigr\rangle_-\circ \t \\
&=\omega_H(h(v_1),h(v_2))\circ\t-\omega_H(w_1,w_2)\circ\s.
\end{align*}
\end{proof}

\begin{rem}
Note that the groupoid geometric data $(\Omega_V, \tilde{\Gamma}, \Omega_H)$ has a \emph{multiplicative} nature:
\begin{itemize}
\item The fiberwise symplectic forms are multiplicative 2-form on the vertical groupoids $\ker\tilde{p}$.
\item The Ehresmann connection $\HOR\tto \Hor$ is a \emph{multiplicative distribution}, since it is a subgroupoid of $TG\tto TE$ over $\Hor\subset TE$.
\item Similarly, equation \eqref{eq:OMEGAH} indicates that $\Omega_H$ is a multiplicative 2-form. There are several ways of expressing this multiplicativity. For example, one may say that for any pair of
composable arrows $(v_1,w_1),(v_2,w_2)\in \Hor^{(2)}$, based at the same composable arrow $(g_1,g_2)\in\G^{(2)}$, one has: 
\[ \Omega_H(m_{\Hor}(v_1,w_1),m_{\Hor}(v_2,w_2))=\Omega_H(v_1,v_2)+\Omega_H(w_1,w_2).\]
One may also say that the composition $\Hor\to \Hor^*\to T^*G$ is a groupoid morphism, where the first map is contraction by $\Omega_H$ and the second one is the inclusion coming from the splitting $TG=\Ver\oplus\Hor$. 
\end{itemize}
\end{rem}

Observing that $\Omega$ is fiber non-degenerate for both $p\circ\s$ and $p\circ \t$, we obtain that:

\begin{cor}
For each $b\in B$, the presymplectic groupoid $(\G,\Omega)$ and the symplectic groupoid $(\G_{E_b}, \Omega_{E_b})$ are Morita equivalent presymplectic groupoids:
\[\SelectTips{cm}{}\xymatrix@C=7pt@R=15pt{                &(\mathcal{P},\Omega_\mathcal{P})\ar[dl]_{\t|_{\mathcal{P}}}\ar[dr]^{\s|_{\mathcal{P}}}& \\
            (\G,\Omega)&                                                     &(\G_{E_b},\Omega_{E_b})}\]
where $\mathcal{P}:=\s^{-1}(E_b)$ and $\Omega_\mathcal{P}:=i_\mathcal{P}^*\Omega$, with $i_\mathcal{P}:\mathcal{P}\hookrightarrow \G$ denoting the inclusion.
\end{cor}

\section{Integration of the Yang-Mills-Higgs phase space}

In \cite{BrFe2} we have proposed an integration procedure for a Yang-Mills-Higgs phase space. This procedure consists in forming a certain Hamiltonian quotient which is hard to make sense for arbitrary coupling Dirac structures, since it will involve an infinite dimensional reduction. In this section, we give a different approach to integrating a Yang-Mills-Higgs phase space.

This new construction of the integration of a Yang-Mills-Higgs phase space $E=P\times_G F$ associated with a triple $(P,G,F)$ and a choice of connection $\theta:TP\to\gg$, involves the following steps:
\begin{enumerate}[(i)]
\item Integration of the Poisson structure on the fiber $(F,\pi_F)$ to a symplectic groupoid $\F\tto F$;
\item Integration of the vertical Poisson structure $\Ver^*$ to a \emph{fibered symplectic groupoid} $\G_V=P\times_G \F\tto E$;
\item Integration of the principal $G$-bundle $P\to B$ to the gauge groupoid $\G(P)\tto B$;
\item The gauge groupoid $\G(P)\tto B$ acts on the fibered groupoid $\G_V\tto E\to B$, yielding a semi-direct product groupoid $\G(P)\ltimes G_V\tto E$.
\item  Finally, the integration of the Yang-Mills phase space is a quotient 
\[ \G(L)=\G(P)\ltimes \G_V/\,\C, \] 
where $\C$ is a certain \emph{curvature groupoid}.
\end{enumerate}
The next paragraphs describe these constructions.

\subsection{Fibered symplectic groupoids} %
\label{sec:fibered:sympl:grpds}

We discuss the first two integration steps above. For this, we recall briefly from \cite{BrFe} a few notions about fibered symplectic groupoids.

\subsubsection{Fibered groupoids}

Let us fix a base $B$. We have the category $\mathbf{Fib}$ of fibrations over
$B$, where the objects are the fibrations $p:E\to B$ and the morphisms
are the fiber preserving maps over the identity.

A \textbf{fibered groupoid} is an internal groupoid in $\mathbf{Fib}$, \emph{i.e.},
an internal category where every morphism is an isomorphism. This means 
that both the total space $\G_V$ and the base $E$ of a fibered groupoid are 
fibrations over $B$ and all structure maps are fibered maps. For instance, the source
and the target maps are fiber preserving maps over the identity:
\[\SelectTips{cm}{}\xymatrix@R=15pt@C=15pt{\G_V\ar[dr]\ar@<0.5ex>[r]\ar@<-0.5ex>[r] & E\ar[d]\\ & B  }\]
It follows that any orbit of $\G_V$ lies in a fiber of $p:E\to B$. In fact, $\G_V|_{E_b}:=(p\circ\s)^{-1}(b)=(p\circ\t)^{-1}(b)$ is a Lie groupoid over $E_b$.

The infinitesimal version of a fibered Lie groupoid $\G_V\tto E\to B$ is a \textbf{fibered Lie algebroid} $A_V\to E\to B$. This means $\pi: A_V\to E$ is a Lie algebroid, the vector bundle projection is map of fibrations:
\[\SelectTips{cm}{}\xymatrix@R=15pt@C=15pt{A_V\ar[dr]\ar[r] & E\ar[d]\\ & B  }\]
and the image of the anchor $\sharp_V$ takes value in the vertical bundle $\Ver\subset TE$. There is an obvious Lie functor from fibered groupoids to fibered algebroids.

A elementary way to obtain a fibered groupoid/algebroid is by using a principal bundle whose structure group acts on a Lie groupoid/algebroid by automorphisms. Then we have the following:
\begin{prop}\label{prop:ass:fib:grpd}
 Consider a principal $G$-bundle $P$ and an action $\Ac:G\rightarrow \Aut(\F)$ of $G$ on a Lie groupoid $\F\tto F$ by Lie groupoid automorphisms.
The associated bundle $P\times_G \F$ carries a natural structure of a fibered Lie groupoid over $E=P\times_G F$. The corresponding fibered Lie algebroid is $P\times_G A(\F)\to E$.
 \end{prop}
 \begin{proof}  The associated bundle $P\times_G \F$ is defined as the space of equivalence classes $[u:a]$ of couples $(u,a)\in P\times \F$ under the relation: 
\[[u:a]=[u':a']\iff (u',a')=(ug^{-1},\Ac_g(a))\text{ for some }  g\in G.\] 
One defines a source and a target map $\s,\t: P\times_G \F \to  P\times_G F$  by: 
\begin{align*}\s[u:a]&:=[u:\s(a)],\\
              \t[u:a]&:=[u:\t(a)],
\end{align*}
and one easily checks that $\s,\t$ are well defined. Then we define a composition by setting $[u':a']\cdot[u:a]=[u, \Ac_g(a')\cdot a]$ where $g$ is the unique element of $G$ such that $u'=ug$. Once checked that it is well defined, observe that there is a more convenient formula for the composition, which is given by:
 \[[u:a']\cdot[u:a]=[u:a \cdot a'].\]
We leave it to the reader to check that these structure maps turn $P\times_G \F$ into a Lie groupoid over $P\times_G F$ with inverse and units given by $[u:a]^{-1}=[u:a^{-1}]$ and ${\bf 1}_{[u:x]}=[u:{\bf 1}_x]$.
\end{proof}

\begin{rem}
As a basic observation, note that each fiber of $P\times_G \F$ comes naturally equipped with the structure of a Lie groupoid over the corresponding fiber of $ P\times_G F$, clearly isomorphic to the model $\F\tto F$.
\end{rem}

\subsubsection{Poisson fibrations}
We now apply these constructions to integrate Poisson fibrations into fibered symplectic groupoids.

\begin{defn}
A \textbf{Poisson fibration} $p:E\to B$ is a locally trivial fiber
bundle, with fiber type a Poisson manifold $(F,\pi_F)$ and with
structure group a subgroup $G\subset \Diff_{\pi}(F)$. When $\pi$ is
symplectic the fibration is called a \textbf{symplectic fibration}.
\end{defn}

 The fibers $E_b:=p^{-1}(b)$ of a Poisson fibration come with an induced Poisson structure $\pi_{E_b}$ that glue to a Poisson structure
$\pi_V$ on the total space of the fibration, so that $\pi_{E_b}=\pi_V|_{E_b}.$

The bivector field $\pi_V$ is \emph{vertical} that is, it takes values in
$\wedge^2\Vert\subset\wedge^2TE$. Hence, the fibers $(E_b,\pi_{E_b})$ 
become Poisson submanifolds of $(E,\pi_V)$.

It is important to distinguish $\pi_V$ as a \emph{vertical} Poisson structure from its underlying Poisson structure on $E$.
In particular, the Lie algebroid structure associated to $\pi_V$ as a vertical Poisson structure is defined on the co-vertical bundle $\Ver^*$, rather than on $T^*\!E$. The corresponding bracket and anchor are given by \eqref{eq:vertical:bracket} and \eqref{eq:vertical:anchor}. Clearly, this is a fibered version of the usual construction, which we formalize as follows:

\begin{defn}
A \textbf{fibered symplectic groupoid} is a fibered Lie groupoid $\G_V$
whose fiber type is a symplectic groupoid $(\F,\omega)$. 
\end{defn}

Therefore, if $\G_V$ is a fibered symplectic groupoid over $B$, then 
$p\circ \s=p\circ\t:\G_V\to B$ is a symplectic fibration, and each symplectic fiber $\G_V|_{E_b}$ 
is a symplectic groupoid over the corresponding fiber $E_b$.

\begin{prop}
The base $E\to B$ of a fibered symplectic groupoid $\G_V\tto E\to B$ has a natural
structure of a Poisson fibration.
\end{prop}

Conversely, a Poisson fibration whose fiber type is an integrable Poisson 
manifold, integrates to a fibered symplectic groupoid. In fact, standard facts about integration of Lie algebroids yield the following (see \cite{BrFe} for details):

\begin{thm}
Let $p:E\to B$ be a Poisson fibration with fiber type $(F,\pi_F)$ an
integrable Poisson manifold. There exists a unique (up to isomorphism)
source 1-connected fibered symplectic groupoid integrating $\Ver^*$. 
\end{thm}

\begin{rem}
Notice that the integration of $\pi_V$ as a Poisson fibration differs from its integration as a mere Poisson manifold. In particular, $\G(\Ver^*)$ has only dimension $2\dim(F)+\dim(B)$.
\end{rem}

\subsection{Action groupoids}
\label{sub:sec:action:grpd}
Next we will discuss steps (iii) and (iv) in the integration of Yang-Mills phase spaces. We describe an action of the gauge groupoid of a principal bundle on an associated fibered groupoid, and the
resulting action groupoid.

\subsubsection{Action of a Lie groupoid on a fibered Lie groupoid}
\label{sub:sec:action:grpd:fibred:grpd}

For a Lie groupoid acting on a $fibered$ Lie groupoid, there is a natural notion of an \emph{action groupoid} \cite{HM} which we recall now. Note that we will only be interested in the case where $\G$ is transitive, although the discussion makes sense in general.

Given a fibered groupoid $\G_V\tto E\overset{p}{\rightarrow} B$, we call {\bf gauge groupoid} the following transitive (infinite dimensional) groupoid:
\[\Gau(\G_V):=\set{\G_V|_{E_b}\xrightarrow{g} \G_V|_{E_{b'}} : g \text{ is a Lie groupoid isomorphism}}\] 
with source $\s(g)=b$, target $\t(g)=b'$, and with the obvious composition.

\begin{defn}
 An {\bf action of a groupoid} $\G\tto B$ {\bf on a fibered groupoid} $\G_V\tto E\rightarrow B$ is a Lie groupoid homomorphism
$\Phi:\G\to \Gau(\G_V).$
\end{defn}
Given such an action, there is an associated groupoid $\G\ltimes\G_V\tto E$ which we will refer to as the semi-direct {\bf action groupoid} associated to the action of $\G\tto B$ on the fibered groupoid $\G_V\tto E\rightarrow B$. As a set, the space of arrows is defined as the fiber product of $\G$ and $\G_V$ over $B$:
\[\G\ltimes\G_V:=\G\, {}_\s\! \times_{p\circ\t} \G_V=\set{ (g,a)\in \G\times \G_V : a\in \G_V|_{E_\s(g)}}.\]
The source and target are given by: $\s(g,a):=\s(a),\ \t(g,a):=\t(\Phi_g(a))$, the units by ${\bf 1}_x=({\bf 1}_{p(x)},{\bf 1}_x)$, the inverses by $(g,a)^{-1}=(g^{-1}, \Phi_g(a)^{-1})$ and the composition is given by the following formula:
\begin{equation}\label{eq:semidirectproduct}
(g_2,a_2)\cdot(g_1,a_1)=(g_2\cdot g_1,\Phi_{g_1}^{-1}(a_2)\cdot a_1).
\end{equation}

When $\G_V$ is trivial (that is, when $\G_V$ has only units) we recover the usual action groupoid  associated to an action of a groupoid on a fibration over its space of objects.

In order to define the infinitesimal counterpart of this action groupoid, for a fibered Lie algebroid $A_V\to E$ we shall denote by:
\[\Der_B(A_V):=\bigl\{D\in\Der(A_V) : \text{symbol of }D \text{ is } p_*\text{-projectable}\bigr\}\]
There is a well defined map $\rho:\Der_B(A_V)\to \X(B)$, $D\mapsto p_*X$, where $X$ denotes the symbol of $D$. We will always assume that $A_V$ is locally trivial, so it follows that $\rho$ is surjective. Notice that $\Der_B(A_V)$ comes with a natural structure of a $C^\infty(B)$-module defined by the formula 
\[ (f.D)(\alpha):=fD(\alpha).\] 
Here, the fact that $f.D$ is indeed an element of $\Der_B(A_V)$ is a consequence of $\sharp_V$ taking values in the vertical bundle. Furthemore $f.D$ has symbol $fX_D$, where $X_D$ denotes the symbol of $D$. Also, one may check that $\Der_B(A_V)$ is a Lie algebra over $C^\infty(B)$ and $\rho$ a morphism of Lie algebras over $C^\infty(B)$.

\begin{defn} An \textbf{action of a Lie algebroid} $A\to B$ \textbf{on a fibered Lie algebroid} $A_V\to E\to B$ is a $C^\infty(B)$-linear Lie algebra morphism $\mathcal{D}:\Gamma(A)\to \Der_B(A_V)$ covering the anchor map that is, such that $\sharp_A=\rho\circ\mathcal{D}$.
\end{defn}

Given such an action,  we denote by $A\ltimes A_V$ the vector bundle $p^*A\oplus A_V$. It comes naturally equipped with the structure of a Lie algebroid over $E$. Since sections of $A$ generate sections of $p^*A$ as a $C^\infty(E)$-module, to define the bracket and anchor we  only need to specify them for sections of the form $(v,\alpha)$, where $v\in\Gamma(A),$ and $\alpha\in\Gamma(A_V)$. The anchor is given by the formula:
\[\sharp(v,\alpha):=X_{\mathcal{D}_v}+\sharp_V(\alpha)\]
where $X_{\mathcal{D}_v}\in \X(E)$ is the symbol of $\mathcal{D}_v$. The bracket takes the form:
\begin{equation*}\left.
\begin{aligned}
\left[v,w\right]_{}&:=\left[v,w\right]_{A}\!\\
\left[\al,\beta\right]_{}&:=\left[\alpha,\beta\right]_{A_V}\!\\
\left[v,\alpha\right]_{}&:=\mathcal{D}_v(\alpha)\\
\end{aligned}\right.
\end{equation*}
for any sections $\al,\beta \in\Gamma(A_V)$ and $v, w \in\Gamma(A)$, seen as sections of $A\ltimes A_V$. Then we extend the bracket and anchor to arbitrary sections of $A\ltimes A_V$.

\begin{defn}
Given an action $\mathcal{D}:\Gamma(A)\to \Der_B(A_V)$ of $A$ on $A_V$, we call \textbf{action Lie algebroid} the Lie algebroid $A\ltimes A_V$ described above.
\end{defn}

\begin{rem} 
\label{rem:algebrd:action:grpd}
If $A_V$ is the Lie algebroid of a fibered Lie groupoid $\G_V$, then $\Der_B(A_V)$ can be thought of as the Lie algebroid of $\Gau(\G_V)$. Then, if  $\Phi:\G\to \Gau(\G_V)$ is an action of $\G$ on a fibered groupoid, differentiation gives a Lie algebroid action $\mathcal{D}:A\to \Der_B(A_V)$. Moreover,  $A\ltimes A_V\to E$  is the Lie algebroid associated with the semi-direct action groupoid $\G\ltimes \G_V\tto E$. 

However, in order for an infinitesimal action $\mathcal{D}:\Gamma(A)\to\Der_B(A_V)$ to integrate to a groupoid action $\Phi:\G(A)\to \Gau(\G(A_V))$, we need to assume that $A$ acts by complete lifts, namely that the symbol of $\mathcal{D}_\alpha$ is a complete vector field on $E$ for any $\alpha \in \Gamma(A)$. 
 \end{rem}

\subsubsection{Actions of principal bundles on fibered Lie groupoids}

Given a principal $G$-bundle $P$, we recall that its {\bf gauge groupoid}, denoted by $\G(P)\tto B$, has spaces of arrows the associated bundle $P\times_{G}P$. Denoting by $[u_2:u_1]$ the equivalence class of a couple $(u_2,u_1)\in P\times P$, the source and target of $\G(P)$ are defined by:
\begin{align*} 
\s([u_2:u_1])&:=q(u_1),\\
\t([u_2:u_1])&:=q(u_2),
\end{align*}
and the composition by: 
\[ [u_4:u_3]\cdot[u_2:u_1]=[u_4g:u_1], \]
where $g$ is the unique element of $G$ such that $u_3g=u_2$. Once checked that the composition is well defined, one may use a more convenient formula, as follows:
 \[[w:v]\cdot[v:u]=[w:u].\]
Inverses and identities are given by $[v:u]^{-1}=[v:u]$ and ${\bf{1}}_x=[u:u]$, for any $x\in B$, and any choice of $u\in P_x$.

The groupoid $\G(P)\tto B$ is transitive and its isotropy groups fit into a Lie group bundle $\Iso_P\to B$. In fact, $\Iso_P$ canonically identifies to the associated bundle $P\times_G G$ and we have:
\[\SelectTips{cm}{}\xymatrix@C=20pt@R=1pt{
               \Iso_P=P\times_G G   \ \ar@{^{(}->}@<-0.5pt>[r] & \G(P)=P\times_G P \\
             [u:h ]   \  \ar@{|->}@<-0.25pt>[r]    & [uh:u ].
}\]
Note that the neutral connected component of the bundle of groups $\Iso_P=P\times_G P$ is $\Iso^\circ_P=P\times_G G^\circ$ and naturally injects in $\G(P)$ by the same formula as above. Here $G^\circ$ denotes the connected component of the identity in $G$.

We will be interested in the $\s$-simply connected goupoid $\widetilde{\G}(P)$ corresponding to $\G(P)$ rather than $\G(P)$ itself. The principal bundle corresponding to $\widetilde{\G}(P)$ has total space the universal cover $\widetilde{P}$ of $P$. When $G$ is connected, the structure group $\bar{G}$ of $\widetilde{P}$ fits into an exact sequence:
\[1\to \im \partial_2 \to \widetilde{G}\to \bar{G} \to 1,\]
\comment{oli: connected components}
where $\partial_2:\pi_2(B)\to \pi_1(G)$ is the boundary operator in the long exact sequence of homotopy groups induced by the projection $P\to B$. In other words, one can always assume that $\pi_1(G)=\im \partial_2$, provided one chooses to work with $\widetilde{\G}(P)$ instead of $\G(P)$.

The Lie algebroid associated to $\G(P)$ is usually denoted by $TP/G$. It is a vector bundle over $B$ whose sections are the $G$-invariant vector fields on $P$, and fits in the Attyiah sequence:
\[\SelectTips{cm}{}\xymatrix@C=15pt{     \ker\sharp  \ \ar@{^{(}->}@<-0.25pt>[r]&       TP/G            \ar@{->>}[r]<-0.25pt> &    TB     .}\]

\begin{prop}\label{prop:action:principal:fibered}
Let $P$ be a principal $G$-bundle and $\Ac:G\rightarrow \Aut(\F)$ an action of $G$ on a Lie groupoid $\F\tto F$ by Lie groupoid automorphisms. There is a natural action of the gauge groupoid $\G(P)$ on the associated fibered Lie groupoid $\G_V:=P\times_G \F$.
\end{prop}
\begin{proof}
We denote the associated fibered Lie groupoid by:
\[\SelectTips{cm}{}\xymatrix@R=15pt@C=0pt{ *+[r]{\G_V\!\!:=}\ar@<0.5ex>[d]\ar@<-.5ex>[d]             &\, P\times_G \F\ar@<0.5ex>[d]\ar@<-.5ex>[d]\\
                                    *+[r]{ E:= }                                 &\, P\times_G F,} \] 
and we define a map $\Phi:\G(P)\to\Gau(\G_V)$ by 
\[\Phi_{[u_2:u_1]}([u:a]):=[u_2g:a],\] 
where $g$ is the unique element of $G$ such that $u_1g=u$. It is easy to see that $\Phi$ is well defined. A more convenient formula for $\Phi$ is:
\[\Phi_{[v:u]}([u:a])=[v:a] \quad (u,v\in P,\ a\in \F).\]
This also shows that $\Phi$ takes values in $\Gau(\G_V)$ and is a groupoid morphism.
\end{proof}

\begin{prop}\label{prop:actiongrpd:formulas} In the conditions of the Proposition \ref{prop:action:principal:fibered}, the  action groupoid $\G\ltimes\G_V\tto E$ identifies with the quotient $(P\times P\times \F)/G$, where $G$ acts on $P\times P \times F$ diagonally:
\[[v:u:a]=[vg:ug:\Ac_g^{-1}(a)] \quad (g\in G).\]
 Under this identification, the structure maps are given as follows:
\begin{itemize}
\item the source and target map are given by:
\begin{align*}
\s[v:u:a]&=[u:\s(a)],\\
\t[v:u:a]&=[v:\t(a)],
\end{align*}
\item the unit at a point $[u:x]$, where $x\in F,\,u\in P$, is given by:
 \[{\bf 1}_{[u:x]}=[u:u:{\bf 1}_x]\] 
\item the inverses are given by the following formula:
 \[[v:u:a]^{-1}=[u:v:a^{-1}].\]
\item and the composition is given by:
\[[w:v:a']\cdot[v:u:a]=[w:u:a'\cdot a].\]
\end{itemize}
\end{prop}

\begin{proof}
According to the construction of the action groupoid in Section \ref{sub:sec:action:grpd:fibred:grpd}, an arrow in $\G(P)\ltimes \G_V$ is a couple $([u_2:u_1],[u:a])$, where $q(u)=q(u_1)$. Since there exists a unique $g\in G$ such that $u_1g=u$, we can always assume that $u_1=u$, and we obtain an identification with $(P\times P\times \F)/G$ by writing the arrows as $([v:u],[u:a])\simeq[v:u:a]$. The formulas for the structure maps then follow from Proposition \ref{prop:action:principal:fibered} and the definition of the action groupoid.
\end{proof}

\subsubsection{Action groupoid of a Poisson fibration} 
Let $E=P\times_G F\to B$ be a Poisson fibration associated with a principal $G$-bundle $p:P\to B$ and an action of $G$ on an integrable Poisson manifold $(F,\pi_F)$ by Poisson diffeomorphisms. The results above show that one obtains an action groupoid as follows.

First, we consider the source connected symplectic groupoid $\F\tto F$ integrating $(F,\pi_F)$. The $G$-action on $F$ by Poisson diffeomorphisms lifts to Lie groupoid action $\Ac:G\to \Aut(\F)$ by groupoid automorphisms (see, e.g., \cite{FeOrRa}). Therefore, according to Propositions \ref{prop:action:principal:fibered} and \ref{prop:actiongrpd:formulas}, there is a natural action of the gauge groupoid $\G(P)\tto B$ on the associated fibered Lie groupoid $\G_V:=P\times_G \F\tto E\to B$, giving rise to an action Lie groupoid $\G(P)\ltimes \G_V\tto E$. 

According to the preceding discussion (see Remark \ref{rem:algebrd:action:grpd}), the Lie algebroid of the action Lie groupoid $\G(P)\ltimes \G_V\tto E$ has underlying vector bundle 
\[ p^*A\ltimes A_V=p^*(TP/G)\ltimes \Ver^*. \]
To determine the bracket and the anchor, we need to find the Lie algebra homomorphism $\mathcal{D}:\Gamma(TP/G)\to \Der_B(\Ver^*)$. Since $\Ver^*$ identifies naturally with the 
associated bundle $\Ver^*=P\times_G\, T^*\!F$ and since the action of $G$ on $T^*\! F$ is naturally lifted from the $G$-action on $F$, $\mathcal{D}$ associates to each $G$-invariant vector field $X$ in $P$
the Lie derivative of the vector field $X_E\in\X(E)$, induced by the natural action on $T^*\!P/G$ on $E$. In other words, $\mathcal{D}_v$ coincides with the Lie derivative of its own symbol:
\[ \mathcal{D}_X(\alpha)=\Lie_{X_E}\al, \quad(X\in\Gamma(TP/G),\, \alpha\in\Gamma(\Ver^*)) \]
where $X_E$ is the projection on $E=P\times_G F$ of the vector field $(X,0)\in\X(P\times F)$. It follows that if $X,Y$ denote $G$-invariant vector fields in $P$ and $\al,\be\in \Gamma(\Ver^*)$, then the anchor of $p^*(TP/G)\ltimes \Ver^*$ is given by the formula:
\begin{equation}
\label{eq:anchor:semi}
\sharp(X,\alpha):=X_E+\pi^\sharp_V(\alpha)
\end{equation}
and the bracket takes the form:
\begin{align}
\label{eq:brackets:semi}
\left[X,Y\right]_{A\ltimes A_V}&:=\left[X,Y\right],\notag\\
\left[\al,\beta\right]_{A\ltimes A_V}&:=\left[\alpha,\beta\right]_{\Vert^*},\!\\
\left[X,\alpha\right]_{A\ltimes A_V}&:=\Lie_{X_E}(\alpha).\notag
\end{align}

\subsection{Integrability of Yang-Mills-Higgs phase spaces}\label{sec:int:YMH2}

\comment{oli: this subsection was in the end of the section 4, I moved it here because most of the material needed was refering to the subsequent Section 5 ! this makes things a bit lengthy before arriving to the second integration of YMH, but I don't think we really have the choice}

We consider now the last steps in the construction of the integration of Yang-Mills-Higgs phase space. So now we assume that we have
\begin{itemize}
\item $p:P\to B$ a principal $G$-bundle;
\item $(F,\pi_F)$ a Poisson manifold;
\item $G\times F\to F$ a hamiltonian $G$-action on $(F,\pi_F)$ with equivariant moment map $J_F:F\to \g^*$.
\end{itemize}
Each choice of a principal connection $\theta:TP\to\gg$ yields a coupling Dirac structure on $E=P\times _G F$.

The fact that the action is hamiltonian implies that the $G$ action on the algebroid $T^*\! F$ is pre-hamiltonian, with pre-moment map (see Appendix A):
\begin{align*}\psi:\g\ltimes F &\longrightarrow   T^*\! F\\
                                     (\xi,m)& \longmapsto \d_m \left\langle J,\xi\right\rangle.
\end{align*}
Therefore, by Theorem \ref{int:ham:action}, $\psi$ integrates to a groupoid morphism:
\[\Psi:G^\circ\ltimes F \to \F,\]
where: 
\[ \F:=\Sigma(F)/\widetilde{\Psi}(\pi_1(G)\ltimes F).\] 
We will assume that $\widetilde{\Psi}(\pi_1(G)\times F)$ is embedded in $\Sigma(F)$, so that $\F$ is smooth. Clearly, $\G_V:=P\times_G \F$ is a symplectic groupoid integrating $\Ver^*=P\times_G T^*\!F$. 

The $G$-action on $F$ lifts to a Lie groupoid action $\Ac:G\to \Aut(\F)$, so we can apply the construction of the previous subsection: we obtain an action groupoid $\G(P)\ltimes\G_V\tto E$.  We now define a completely intransitive subgroupoid of $\G(P)\ltimes\G_V$, that we think as a model for the homotopy relation in the phase space.

\begin{defn}
The \textbf{curvature subgroupoid}, denoted by $\mathcal{C}\tto E$, is the subgroupoid $\C\subset \G(P)\ltimes\G_V$ given by:
\[\C:=\graph(\Psi_P\circ i) \subset \G(P)\ltimes\G_V,\]
where $\Psi_P:\Iso_P^\circ\times_B E\to\G_V$ is obtained by fibrating $\Psi:G^\circ\ltimes F \to \F$ along $P$ and $i:\Iso^\circ\to \Iso^\circ$ denotes the inversion.
\end{defn}

With the notations of Proposition \ref{prop:actiongrpd:formulas}, the curvature groupoid $\C$ is given explicitly by:
\begin{equation}
\label{curvaturegroupoid}
\C:=\set{[uh^{-1}:u:\Psi(h,x)]\in \G(P)\ltimes\G_V : u\in P,\, h\in G^\circ,\, x\in F}.
\end{equation}
This allows to prove the following.

\begin{prop}  The curvature groupoid $\C\tto E$ is a wide, normal, completely intransitive subgroupoid of $\G(P)\ltimes \G_V$. 
\end{prop}
\begin{proof}
The result follows using expression \eqref{curvaturegroupoid} for $\C$, the compositions rules in Proposition \ref{prop:actiongrpd:formulas} and the equation \eqref{eq:int:ham:action1} in Theorem \ref{int:ham:action}.
The fact that $\C$ is a subgroupoid is rather straightforward. In order to see that it is normal, we pick any  $[u h^{-1}:u:\Psi(h,x)]\in\C$ and $[v:u:a]\in\G(P)\ltimes \G_V$ which are composable, i.e., such that $x=\s(a)$, and we find that:
\[ [v:u:a]\cdot [u h^{-1}:u:\Psi(h,x)]\cdot[v:u:a]^{-1}=[v h^{-1}:v:\Psi(h,x)],\]
is an element in $\C$.
\end{proof}

Finally, putting all together, we conclude that:

\begin{thm} 
\label{thm:Yang:Mills:integration}
Let $(P,G,F)$ be a classical Yang-Mills-Higgs setting and $\theta:TP\to \g$ a principal connection.  Let $L$ be the corresponding coupling Dirac structure on $E=P\times_G F$ and
assume that:
\begin{enumerate}[(i)]
\item the Poisson manifold $(F,\pi_F)$ is integrable;
\item the groupoid $\widetilde{\Psi}(\pi_1(G)\times F)$ is embedded in $\Sigma(F)$.
\end{enumerate}
Then the quotient groupoid $\G(P)\ltimes {\G_V} /\, \C$ integrates $(E,L)$.
\end{thm}

\begin{proof}
As we saw above, the Lie algebroid of $\G(P)\ltimes {\G_V}$ is given by $A\ltimes A_V$, where $A=TP/G$ and $A_V=\Ver^*$. Furthermore, the principal connection induces a splitting of the Attiyah sequence, and we have an identification $TP/G\simeq TB\oplus \ker \sharp_{A}$. 
With this identification, the Lie algebroid $A_\C$ of $\C$ lies in $TP/G\ltimes \Ver^*\simeq (TB\oplus \ker\sharp_{A})\ltimes \Ver^*$  as: 
\[ A_\C=\set{\bigl(0,\xi,-\psi(\xi)\bigr)\in (TB\oplus \ker \sharp_A )\ltimes \Ver^* : \xi \in \ker \sharp_A},\]
and the quotient $(TB\oplus \ker \sharp_A )\ltimes \Ver^*/A_\C$ identifies with $TB\times_B \Ver^*$, with canonical projection given by $\pi(X,\xi,\al)=(X,\al+\psi(\xi))$. It now follows from expressions \eqref{eq:anchor:semi} and \eqref{eq:brackets:semi} for the anchor and the brackets that the Lie algebroid structure on $A\ltimes\Ver^*$ descends to a Lie algebroid structure on $TB\times_B A_V$ whose brackets and anchor are the same as those given in Proposition \ref{prop:splitbrackets}. Hence $A\ltimes \Ver^*\!/A_\C$ and $L$ are isomorphic as Lie algebroids.

For the smoothness of the quotient, we observe that $\G(P)\ltimes {\G_V} /\, \C$ can also be thought of as an ``associated bundle" $\G(P)\ltimes_{\Iso^\circ_P} \G_V$. Indeed, there is an action of the \emph{bundle of} Lie groups $\Iso^\circ_P$ on $\G(P)\ltimes {\G_V}$ which can be described as follows. On the one hand, the Lie groupoid morphism $\Psi$ induces an action $\lambda$ of $G^\circ$ on $\F$ by left multiplication: $\lambda_h(a):=\Psi(h,x) \cdot a,$ where $h\in G^\circ$, and $x:=\s(a))$. Fibering along $P$, we obtain an action of the bundle of Lie groups $\Iso_P^\circ$ on $\G_V$:
\[\lambda^P_{[u:h]}([u:a]):=[u:\Psi(h,x)\cdot a].\]
Here we use the identification $\Iso^\circ_P\simeq P\times_G G^\circ$, to write an element of $\Iso^\circ_P$ as a pair $[u:h]$. Note that this action is well defined by \eqref{eq:int:ham:action2}. On the other hand, $\Iso_P^\circ$ acts on $\G(P)$ by right multiplication, which is a proper and free action. The two actions together give a proper and free action of $\Iso^\circ_P$ on  $\G(P)\ltimes \G_V$:
\[ g\cdot (b,a):=(bg^{-1}:\lambda^P_g(a)). \]
and the quotient is the ``associated bundle" $\G(P)\ltimes_{\Iso^\circ_P} \G_V$. 

We claim that  $\G(P)\ltimes_{\Iso^\circ_P} \G_V$ can be identified with $\G(P)\ltimes {\G_V} \!/ \C$. This follows by observing that any 
$g\in \Iso_P^\circ$ can be written as $g=[u:u h]\in \G(P)$  so that (see Theorem \ref{int:ham:action}):
\[g\cdot(b,a)=(bg^{-1},\lambda_g^P(a))=(b,a)\cdot c\] 
where $c:=([u h^{-1}:u:\Psi(h)])\in \C$. The assignment $c\leftrightarrow g$ being 1-1, one concludes that the two quotients coincide.
\end{proof}

\begin{rem}
Consider the Hopf fibration $P=\mathbb{S}^3\to \mathbb{S}^2$, seen as a $\mathbb{S}^1$-principal bundle,  and $F=\mathbb{R}$ acted upon trivially  with momentum map $f:F\to \mathbb{R}$ any smooth function, as in Example \ref{ex:Hopf}. Then the second condition in Theorem \ref{thm:Yang:Mills:integration} fails if $f$ has a critical point as explained in Example \ref{ex:non-smooth-quotient}.
\end{rem}

Theorem \ref{thm:Yang:Mills:integration} shows that the groupoid structure of $\G(L)$ does not depend on the choice of principal $G$-bundle connection. In other words, two coupling Dirac structures associated with Yang-Mills data with the same principal $G$-bundle and hamiltonian $G$-action, but different principle bundle connections give rise to the same Lie groupoid. Note, however, that the presymplectic forms will be distinct, as it is clear from their geometric data given in Proposition \ref{prop:int:geom:data}. 

We can also give an explicit description of the presymplectic form $\Omega$ on $G(L)$, as follows. First, we use Proposition \ref{prop:actiongrpd:formulas} to identify $\G(P)\ltimes {\G_V}\simeq (P\times P\times \F)/G$. We then construct a presymplectic form $\widetilde{\Omega}$ on $\G(P)\ltimes {\G_V}$: we have a closed two form on $P\times P\times \F$ given by
\[ \widetilde{\Omega}:=p_\F^*\Omega_\F+\d\langle \theta,\overline{\mu}\rangle_1-\d\langle \theta,\overline{\mu}\rangle_2,\] 
where  $p_\F:P\times P\times \F\to \F$ is the projection, and $\d\langle \theta,\bar{\mu}\rangle_i$ denotes the closed 2-form in $\Omega^2(P\times P\times \F)$ obtained by differentiating the 1-form $\alpha_i\in \Omega^1(P\times P\times \F)$ given by
\[ \alpha_i|_{(u_1,u_2,g)}(v_1,v_2,w):= \langle \theta|_{u_i}(v_i),\overline{\mu}(g)\rangle. \]
Here $\overline{\mu}:\F\to\gg^*$ denotes the moment map for the lifted $G$-action on $\F$, so that $\overline{\mu}=\mu\circ\t-\mu\circ\s$.
One checks easily that the closed 2-form $\widetilde{\Omega}$ is basic for the the $G$-action on $P\times P\times \F$, so it descends to a multiplicative 2-form in the quotient $(P\times P\times \F)/G\simeq \G(P)\ltimes {\G_V}$. 

Finally, one checks that resulting multiplicative 2-form on $\G(P)\ltimes {\G_V}$ further descends to the quotient $\G(P)\ltimes {\G_V}/\C$, giving a closed, multiplicative 2-form $\Omega_\G$ satisfying the non-degeneracy condition \eqref{eq:non-degeneracy}. A more or less tedious computation shows that the target map $\t:(\G(L),\Omega_\G)\to (E,L)$ is a forward Dirac map. Summarizing this discussion, we have:

\begin{cor}
Under the conditions of Theorem \ref{thm:Yang:Mills:integration}, the presymplectic form on the groupoid $G(L)=\G(P)\ltimes {\G_V}/\,\C$ is the quotient of the closed 2-form:
\[ \widetilde{\Omega}:=p_\F^*\Omega_\F+\d\langle \theta,\overline{\mu}\rangle_1-\d\langle \theta,\overline{\mu}\rangle_2.\]
\end{cor}

The integrability conditions in Theorem \ref{thm:Yang:Mills:integration} can be made more explicit. On the one hand, the integrability of the fiber type $(F,\pi_F)$ follows under the general theory developed in \cite{CrFe1}, and can be expressed in terms of monodromy maps $\partial: \pi_2(S,x)\to\G(\gg_x)$, where $S$ is the symplectic leaf of $F$ through $x$ and $\gg_x=\ker\pi_F^\sharp|_x$ is the isotropy Lie algebra at $x$. On the other hand, condition (ii) can be treated by the same methods as in \cite[Section 4.3]{BrFe2}, and one gets another monodromy type map $\pi_1(G)\to \F_m$ controlling (ii). This will be treated elsewhere.

\section{Integration of coupling Dirac structures II}
\label{gpd:construction}
A general coupling Dirac structure may not come from a principal bundle with structure group a finite dimensional Lie group. For instance, this is the case if the holonomy group induced by the connection (\emph{i.e.}, the group spanned by the holonomy along loops in the base) is not a finite dimensional subgroup of the Poisson automorphisms of the fiber. In such cases, one needs a formulation of the construction of Section \ref{sec:int:YMH2} which avoids infinite dimensional reductions. In this section, we will take advantage of the fact that $L$ fits into a Lie algebroid extension, to reformulate the construction given in Section \ref{sec:int:YMH2}, without any mentioning to to these infinite dimensional group quotients. We follow the ideas of \cite{Br} in order to describe $L$-paths and $L$-homotopies. 

Recalling that $\Ver^*=\graph(\pi_V)$, we know that $L$ is a Lie algebroid extension:
\[\SelectTips{cm}{}\xymatrix@C=15pt{  \Ver^*  \ \ar@{^{(}->}@<-0.25pt>[r]&   L                \ar@{->>}[r]<-0.25pt> &   TB ,}\]
which splits. The notion of holonomy makes sense for any Lie algebroid extension with a splitting (see \cite[Sec.\,2.1]{Br}), and in our situation, given a $TB$-path $\dot{\gamma}_B\in P(TB)$, the holonomy is a Lie algebroid morphism:
\[\Phi_{\gamma_B}:\Ver^*|_{E_{\gamma_B(0)}}\to \Ver^*|_{E_{\gamma_B(1)}}.\]
It will be useful to restrict $\gamma_B$ to a path $[0,t]\to TB$, where $t\in[0,1]$. The corresponding holonomy will then be denoted by:
\[\Phi^{\gamma_B}_{t,0}:\Ver^*|_{E_{\gamma_B(0)}}\to \Ver^*|_{E_{\gamma_B(t)}}.\]

In the case of a coupling Dirac structure, there is another notion of holonomy to be taken into account, namely, the one induced by the \emph{usual} Ehresmann connection $\Hor$. Given a path $\gamma_B:[0,1]\to B$ it gives rise to a holonomy map
\[ \phi^{\gamma_B}:E_{\gamma_B(0)}\to E_{\gamma_B(1)}. \] 
Again, restricting $\gamma_B$ to a path $[0,t]\to B$ the corresponding holonomy will be denoted 
\[ \phi^{\gamma_B}_{t,0}:E_{\gamma_B(0)}\to E_{\gamma_B(t)}.\]  

The two holonomies are related in a simple way:

\begin{prop}
\label{prop:holonomies}
 The holonomy $\Phi^{\gamma_B}_{t,0}$ induced by the connection $\graph{(\omega_H)}$ on $L$ is related to the holonomy 
 $\phi^{\gamma_B}_{t,0}$ induced by $\Hor$ on $TE$ by: 
 \[\Phi^{\gamma_B}_{t,0}=(\phi_{0,t}^{\gamma_B})^*.\]
\end{prop}

\begin{proof}
The result follows directly from the identification $\Ver^*=\graph(\pi_V)$ and from the particular form of the bracket \eqref{eq:splitbracket2}. 
\end{proof}

Recall that a Lie algebroid extension is called a \textbf{fibration} whenever the Ehresmann connection is complete \cite{BZ}. It follows from Proposition \ref{prop:holonomies} that \eqref{eq:Dirac:extension} is a fibration whenever the Ehresmann connection $\Hor$ is complete. In the sequel, we will always assume that it is the case.

\subsection{Splitting $L$-paths and $L$-homotopies}
\label{splitting:paths}

 We see from Proposition \ref{prop:Dirac:extension} that any $L$-path $a$ over $\gamma:=p_L\circ a$ decomposes uniquely as a sum:
               \begin{equation}\label{eq:splitpath0} a(t)=h^*(\dot{\gamma}_B(t))_{\gamma(t)}+a_V(t)\end{equation}
where $\gamma_B:=p\circ\gamma$.
In this decomposition, neither $t\mapsto h^*(\dot{\gamma}_B(t))_{\gamma(t)}$, nor $t\mapsto a_V(t)$ is an $L$-path in general. However, one may wonder if it is possible to somehow 'split' paths in $P(L)$ into horizontal and vertical parts. The next proposition shows how this can be done:

\begin{prop}[Splitting $L$-paths]\label{splitpath}
Let $L$ be a coupling Dirac structure on a fibration $p:E\to B$. If the associated connection $\Gamma$ is complete, then there is an isomorphism of Banach manifolds:
\[\begin{array}{ccc}
 P(L)   &\longrightarrow& P(TB)\, {}_\s\!\times_{\,\t\circ p} P(\Ver^*)\\
    a      &\longmapsto&(\dot{\gamma}_B,\tilde{a}),
\end{array}\]
where the couple $(\dot{\gamma}_B,\tilde{a})$ is defined in the following way:
\begin{equation}
\label{gammaB}                     
\dot{\gamma}_B :=\d p\circ\sharp a, \quad \tilde{a}_t:=a_V(t)\circ \d\phi^{\gamma_B}_{t,0},
\end{equation}
where $\phi^{\gamma_B}_{t,0}:E_{\gamma_B(0)}\to E_{\gamma_B(t)}$ denotes the holonomy along $\gamma_B$.
\end{prop}
\begin{proof}
This follows from \cite[Prop. 4.1]{Br} and Proposition \ref{prop:holonomies}.
\end{proof}

One should think of the couple $(\dot{\gamma}_B,\tilde{a})$ as a concatenation of $L$-paths of the form $h^*(\dot{\gamma}_B)\cdot\tilde{a}$. Here, $h^*(\dot{\gamma}_B)$ denotes the $L$-path defined by:
\begin{equation}\label{eq:path:horizontal:lift}h^*(\dot{\gamma}_B)(t):=h^*(\dot{\gamma}_B(t))_{\phi_{t,1}^{\gamma_B}(y)},
\end{equation}
 where $y=\s(a)$. Notice that the $L$-path \eqref{eq:path:horizontal:lift} is different from the horizontal component appearing in \eqref{eq:splitpath0} since the base paths are different. In particular, $h^*(\dot{\gamma}_B)$ as defined in \eqref{eq:path:horizontal:lift} is \emph{always} a $L$-path by construction.

 Then Proposition \ref{splitpath} can be illustrated in a simple way as follows:\vspace{5pt}
\[\xymatrix@R=10pt{
                    & \phi_{\gamma_B}^{-1}(y)  \ar@/^0.5pc/[rr]^{\, h^*(\dot{\gamma}_B)\, }&  &    *+[r]{ y}                                                                                                \\
                    &   &                             &                                                                       &L\ar[dd]_{} \\
                    & x \ar@/^0.5pc/[uu]^{\tilde{a}} \ar@/^0,5pc/@{->}[rruu]_{a}&
                        &                                             \\
                    & \gamma_B(0)   \ar@/^0.5pc/[rr]^{\, \dot{\gamma}_B\, }&
                        &     \gamma_B(1)                      &                                         TB    }\]
In fact, it can be proved that $a$ is $L$-homotopic to the concatenation $h^*(\dot{\gamma}_B)\cdot\tilde{a}$. However, for the sake of simplicity, in this work we shall simply think of the map $a\mapsto (\tilde{a},\dot{\gamma_B})$ as an mere identification.

Recall that for any $A$-path $a$, its \textbf{inverse path} is the $A$-path $a^{-1}$ defined by $a^{-1}(t):=-a(1-t)$.
Using Proposition \ref{splitpath}, one can express the concatenation and inverses of $L$-paths as follows

\begin{prop}\label{prop:split:conc}
Under the isomorphism of Proposition \ref{splitpath}, given two composable $L$-paths $a\simeq (\dot{\gamma}_B,\tilde{a})$ and $b\simeq (\dot{\delta}_B,\tilde{b})$, their concatenation is
\[(\dot{\delta}_B,\tilde{b})\cdot(\dot{\gamma}_B,\tilde{a}):=\bigl(\dot{\delta}_B\cdot\dot{\gamma}_B,\Phi_{\gamma_B}^{-1}(\tilde{b})\cdot \tilde{a}\bigr).\]
Moreover, the inverse path $a^{-1}$ of $a$ is: 
\[(\dot{\gamma}_B,\tilde{a})^{-1}:=\bigl(\dot{\gamma}_B^{-1},\Phi_{\gamma_B}(\tilde{a})^{-1}\bigr).\]
\end{prop}
\begin{proof}
The result follows directly from \eqref{gammaB} and from the fact that the holonomy commutes with taking concatenation and inverse of $A$-paths.
\end{proof}

Notice the analogy between the formula for concatenation in the previous proposition and formula \eqref{eq:semidirectproduct} for the product in the action groupoid. In fact, if one thinks of $P(TB)$ as a groupoid over $B$, then the holonomy gives an action of $P(TB)$ on $P(\Ver^*)$ similarly to the action of $\G(P)$ on $\G_V$ discussed in Section \ref{sub:sec:action:grpd:fibred:grpd}. For this reason, one may think of the fibered product $P(TB)\, {}_\s\!\times_{\t\circ p} P(\Ver^*)$ as a semi-direct product  $P(TB)\ltimes P(\Ver^*)$. 

In general, because of the presence of curvature, there is no action of the fundamental groupoid $\Pi(B)$ on $P(\Ver^*)$. However, holonomy along any path $\gamma_B\in P(B)$ is a Lie algebroid morphism $\Phi_{\gamma_B}:\Ver^*|_{E_{\gamma_B(0)}}\to \Ver^*|_{E_{\gamma_B(1)}}$. Hence, it integrates to a groupoid morphism:
 \[\Phi_{\gamma_B}:\G(\Ver^*)|_{E_{\gamma_B(0)}}\to \G(\Ver^*)|_{E_{\gamma_B(1)}},\]
that we still denote  by $\Phi_{\gamma_B}$. Here, $\G(\Ver^*)$ denotes the Weinstein groupoid of $\Ver^*$. Finally, notice that the formulas in Proposition \ref{prop:split:conc} still make sense when replacing $\Ver^*$-paths by their homotopy classes,  therefore we will denote by $P(TB)\ltimes \G(\Ver^*)$ the fibered product $P(TB)\, {}_\s\!\times_{\t\circ p} \G(\Ver^*)$. 

The identification given by Proposition \ref{splitpath} also allows to describe $L$-homotopies in a nice way:

\begin{thm}\label{thm-2cocycle} Let $L$ be a coupling Dirac structure on $E\to B$. The source 1-connected groupoid $\G(L)$ integrating $L$ naturally identifies with equivalence classes in $P(TB)\ltimes_B \G(\Ver^*)$ under the following relation:
 \begin{itemize}
 \item $(\gamma_0,g_0)\sim(\gamma_1,g_1)$ if and only if $\exists$ a homotopy $\gamma_B:I\times I\to B$, $(t,\epsilon)\mapsto \gamma^\epsilon_B(t)$ between $\gamma_0$ and $\gamma_1$, such that $g_1=\partial (\gamma_B,\t(g_0)).g_0.$
 \end{itemize}
Here, $\partial(\gamma_B, x_0)$ is the element in $\G(\Ver^*)$ represented by the $\Ver^*$-path: 
\begin{equation}\label{partial:path}\epsilon\longmapsto (\d_V)_{\tilde{\gamma}^\epsilon}\Bigl(\,\int_0^1 (\phi^{\gamma_B^\epsilon}_{s,0})^* \omega_H(\gamma_B)_{s,\epsilon}\, \d s\Bigr)\in \Ver^*_{\tilde{\gamma}(\epsilon)} 
\end{equation}
where $\tilde{\gamma}(\epsilon):=\Phi^{-1}_{\gamma_B^\epsilon}\circ \Phi_{\gamma_B^0}(x_0)$ and:
\[\omega_H(\gamma_B)_{s,\epsilon}:=\omega_H\Bigl(\textstyle{h\bigl(\frac{\d\gamma_B}{\d t}(s,\epsilon)\bigr),h\bigl(\frac{\d\gamma_B}{\d\epsilon}(s,\epsilon)\bigr)}\Bigr)\in C^\infty(E_{\gamma_B(s,\epsilon)}).\]
\end{thm}
\begin{proof}
The proof follows immediately from the results in \cite[Sec.\,4]{Br}
\end{proof}

One may illustrate the homotopy condition appearing in Theorem \ref{thm-2cocycle} in the following way:\vspace{10pt}
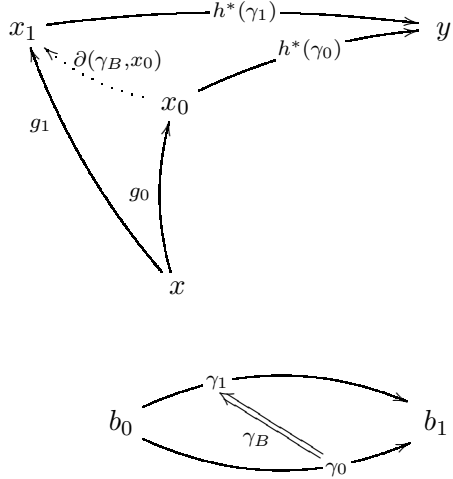
\begin{align*}
&\xymatrix@R=15pt@C=15pt@l{
                                   &&    &&*+[r]{x\quad\quad\quad\quad\quad }&                   \\
																	                                                    \\
                                   &&   & &      \\
																	 &&   & &{x_0}\ar@{<-}@/_0,5pc/[uuu]_{\ g_0}&                                                                       \\
                                   &\ y  \ar@{<-}@/_0,5pc/[rrrr]|>>>>>>>>>>>>{\,h^*(\gamma_1)\,}\quad
																	           \ar@{<-}@/_0,5pc/[rrru]|{\,h^*(\gamma_0)\,}
																						 &&&& x_1  \ar@{<-}@/_0,5pc/[uuuul]_<<<<<<<<<{\ g_1}
																						            \ar@/_0,5pc/@{<..}[ul]^{\partial(\gamma_B, x_0)}       }\\
&\xymatrix{                    \\  & \ \  b_0   \ar@{->}@/_1,5pc/[rrr]|>>>>>>{\,\overset{\ }{\gamma_0}\,}="d" 
                                                \ar@{->}@/^1,5pc/[rrr]|<<<<<<{\,\gamma_1\,}="e"
																							&\ &\ &     \quad {b_1}\quad\quad                  \ar@{=>}^-{\gamma_B}"d";"e"-<-1pt,4pt> }
\end{align*}
\vspace{10pt}

\begin{ex}\label{flat} 
The case where $\pi_V$ is the trivial Poisson structure occurs, for instance, if $L$ is the restriction of a regular Dirac structure to a tubular neighborhood $E\to B$ of one of its leaves $B$. Then $\G(\Ver^*)$ is a bundle of Lie groups that identifies with $\Ver^*$  with its additive structure. Furthermore, it follows from the curvature identity \eqref{eq:curvature} that the connection is flat, so we have a genuine action of the fundamental groupoid $\Pi(B)$ on  $\Ver^*$. Up to a cover of $B$, we may assume that $E$ is trivial as a representation of $\Pi(B)$.  This means that $E$ can be identified with $B\times F$ in such a way that the holonomy along any path is the identity:
\[\phi_{s,0}^{\gamma}=\text{id}_F:\{\gamma(0)\}\times F\longrightarrow\{\gamma(1)\}\times F. \]
It follows that the horizontal and vertical distributions are respectively given by $\Hor=TB\times F$ and $\Ver=B\times TF$ in the decomposition $TE=TB\oplus TF$. Hence $\omega_H$ can be seen as a family of $2$-forms on $B$ parametrized by $F$, and the  leaves of $L$ are of the form $B\times\{x\}$ with presymplectic form $\omega_H|_{B\times\{x\}}$, where $x\in F$.

The homotopy condition appearing in Theorem \ref{thm-2cocycle} can then be expressed as follows: two elements $(\gamma_0,g_0)$ and $(\gamma_1,g_1)$ in $P(TB)\times T_{x_0}^*\!F$ are homotopic if and only if there exists a $TB$-homotopy $\gamma_B:I^2\mapsto B$ between $\gamma_0$ and $\gamma_1$ such that:
\[g_1-g_0=(\d_V)_{x_0}\int_{\gamma_B} \omega_H,\]
where we integrate $\omega_H$ along $\gamma_B$ as a $2$-form with values in $C^\infty(F)$. In order to obtain the above formula, we simply replace $\phi_{0,s}^{\gamma_B^\epsilon}$ by $\text{id}_F$ in \eqref{partial:path} and then we use the fact that, $T^*\!F$ being a bundle of abelian groups, any path in $T_{x_0}^*\!F$ can be represented by a constant paths. This amounts in \eqref{partial:path} to average with respect to the $\epsilon$ variable. In particular, when $F=\mathbb{R}$, we recover the leafwise prequantization Lie algebroids and the homotopy condition appearing in \cite{Cr}. 

Notice also that, applying the resulting $1$-form to a vector $X_{x_0}\in TF$, one gets a geometrical interpretation of $g_1-g_0$ as the variation of the presymplectic area of $\gamma_B$ in the vertical directions:
\[\Bigl\langle\,\d_V\int_{\gamma_B} \omega_H\,, X_{x_0}\,\Bigr\rangle=\left.{\frac{\d}{\d t}}\right|_{t=0}\,\int_{\gamma_B\times\{\phi_t^{X}(x_0)\}} \omega_H,\]
where $\phi_t^X$ is the flow of any vector field $X\in\X(F)$ extending $X_{x_0}$. 
\end{ex}

\begin{rem}
The construction given in Theorem \ref{thm-2cocycle} can be interpreted as an infinite dimensional analogue of the construction given in Section 4 of the groupoid integrating a Yang-Mills phase space. 

For this interpretation, consider the Poisson frame bundle of Section \ref{sub:sec:examples:5}, so that we can view our coupling as an infinite dimensional Yang-Mills phase space. One needs first to reduce the structure group, from the group of Poisson diffeomorphisms between a fixed fiber $E_{b_0}$ and any other fiber, to the subgroup generated by the holonomy transformations  $\Phi_{\gamma_B}$ along any path $\gamma_B\in P(B)$, with $\gamma_B(0)=b_0$. If $P\to B$ denotes the resulting principal bundle, then one can ``identify" the corresponding gauge groupoid $\Gau(P)=P\times_B P$ with the ``groupoid" $P(TB)$. Moreover, the equivalence relation $\sim$ of Theorem \ref{thm-2cocycle} can be viewed as the equivalence relation associated with the corresponding curvature groupoid. 
\end{rem}

\subsection{The Monodromy Groupoid}\label{monodromy:grpd}  
We now use the constructions of \cite{Br,CrFe1} in order to obtain the obstructions to integrability of a coupling Dirac structure $L$.

Consider the short exact sequence of Lie algebroids:
\begin{equation}\label{exact:alg}
\SelectTips{cm}{}\xymatrix@C=20pt{  \Ver^*  \ \ar@{^{(}->}@<-0.25pt>[r]&   L                \ar@{->>}[r]<-0.25pt> &   TB      ,}
\end{equation}
We obtain by integration the sequence of groupoid morphisms: 
\begin{equation}\label{exactgrpd}
\SelectTips{cm}{}\xymatrix@C=20pt{ \G(\Ver^*) \ \ar[r]^-{j}&    \G(L)       \ar[r]^-{q} &   \Pi(B)     ,}
\end{equation}
where $\Pi(B)$ denotes the fondamental groupoid of $B$. Recall $j$ and $q$ are defined at the level of paths:
\begin{align*}
j([\tilde{a}]_{\scriptscriptstyle V})&:=[i \circ a]_{\scriptscriptstyle{L}}, \\
q([a]_{\scriptscriptstyle L})&:=[{p}_*\circ\sharp(a)]_{\scriptscriptstyle{TB}}. 
\end{align*}
for any $\Ver^*$-path $\tilde{a}:I\to \Ver^*$ and any $L$-path $a:I\to L$. Although the sequence \eqref{exact:alg} is exact, the sequence \eqref{exactgrpd}  might not be exact anymore, pointing out a lack of exactness of the integration functor. One can always ensure of the right exactness as follows:
\begin{prop}
Let $L$ be a coupling Dirac structure whose induced Ehresmann connection is complete. Then the sequence \eqref{exactgrpd} is surjective at $\Pi(B)$, and exact at $\G(L)$.
\end{prop}
\begin{proof}
One may see that $q$ is surjective, provided the connection is complete, by observing that, given $[\dot{\gamma}_B]_{TB}\in \Pi(B)$, the element $[h^*(\dot{\gamma}_B)]\in\G(L)$, defined by \eqref{eq:path:horizontal:lift} maps to $[\dot{\gamma}_B]$.
 
For the exactness at $\G(L)$ one observes that, by the definition, the elements of $\ker\, q$ are represented by $L$-paths whose projection on $TB$ is a contractile loop. Therefore, the inclusion $\im\  j\subset \ker\ q$ is obvious. Conversely, given an element  $[a]_L\in\ker\ q$, represented by some $L$-path $a$, we see that $a\sim(\tilde{a},\dot{\gamma}_B)$, under the identifications of Proposition \ref{splitpath}, where $\gamma_B$ is a contractible loop based at some $b\in B$. Consider a contraction $\gamma^\epsilon_B:I^2\to B$ between  $\gamma_B$ and the trivial path $0_b$. Then by Theorem \ref{thm-2cocycle}, we see that $(\tilde{a},\dot{\gamma}_B)$ is $L$-homotopic to $(\partial(\gamma_B)\cdot\tilde{a}, 0_b)$. Since $(\partial(\gamma_B)\cdot\tilde{a}, 0_b)$ represents a $\Ver^*$-path, we conclude that $[a]_L\in\im \ j$, as claimed. 
\end{proof}

It follows that \eqref{exactgrpd} can only fail to be exact because of the lack of injectivity of $j$. In order to measure this failure, we introduce the following:

\begin{defn}
The {\bf monodromy groupoid} associated with the fibration is the kernel of $j:\G(\Ver^*)\to \G(L)$, denoted by $\mathcal{M}$.
\end{defn}

Obviously, by construction, we have an exact sequence of groupoids:
\[
\SelectTips{cm}{}\xymatrix@C=15pt{ \M  \ \ar@{^{(}->}@<-0.25pt>[r]&  \G(\Ver^*)            \ar@{->>}[r]<-0.25pt> &  \ker\ q     .}
\]
and we can replace \eqref{exactgrpd} by the exact sequence of groupoids:
\[\SelectTips{cm}{}\xymatrix@C=15pt{ \G(\Ver^*)/\M  \ \ar@{^{(}->}@<-0.25pt>[r]&  \G(\Ver^*)            \ar@{->>}[r]<-0.25pt> &  \Pi(B)    .}\]
The kernel of this sequence $\ker q=\G(\Ver^*)/\M$ is a bundle of groupoids with typical fiber the neutral component of the restricted groupoid $\G(L)|_{E_{b_0}}$ to a fiber $E_{b_0}$. In particular, we see that if $\G(L)$ is integrable, then the monodromy groupoid $\M$ must be embedded in $\G(\Ver^*)$.

It remains to relate $\M$ to the global data associated with $L$ on $E$. 

\begin{thm}\label{thm:transgression}
Consider a coupling Dirac structure $L$ on a fibration $E\to B$, and assume that the induced Eheresmann connection is complete. Then there exists a homomorphism 
\[\partial:\pi_2(B)\times_{B} E\to \G(\Ver^*),\]
that makes the following sequence exact:
\[ \dots\to \pi_2(B)\ltimes E\to \G(\Ver^*)\to \G(L)\to\Pi(B).\]
\end{thm}
In other words, Theorem  \ref{thm:transgression} states that the monodromy groupoids of the fibration coincide with the image of the transgression map $\M=\im \partial|_{\pi_2(B)}$.
\begin{proof}
The map $\partial$ in Theorem \ref{thm-2cocycle}, when restricted to a sphere in $B$ based at some $b\in B$ (seen as a map $\gamma_B:I^2\to B$ such that $\gamma_B(\partial I^2)=\{b\}$) is independent of its homotopy class (see \cite{Br}). Then it follows from \cite{Br,BZ} that the restriction of the map $\partial$ to $\pi_2(B)$ corresponds precisely to the transgression map.
\end{proof}

Note the analogy between the monodromy groupoid described above and the monodromy groups that measure the integrability of an algebroid \cite{CrFe2}. In fact, when $E$ is a tubular neighborhood of a leaf $B\subset E$ in a Dirac structure, the restriction $\M|_B$ coincides, by construction, with the usual monodromy groups along $B$.

Finally, we can relate  the monodromy groupoid of a coupling Dirac structure with the problem of integrability as follows:

\begin{thm} \label{th:integrability}
Let $L$ be a coupling Dirac structure on $E\to B$ and assume that the associated connection $\Gamma$ is complete. Then $L$ is an integrable Lie algebroid \emph{iff} the following conditions hold:
\begin{enumerate}[(i)]
\item the typical Poisson fiber $(E_x,\pi_V|_{E_x})$ is integrable,
\item the injection $\mathcal{M}\hookrightarrow \G(\Ver^*)$ is an embedding.
\end{enumerate}
\end{thm}

\begin{proof}
First it is easily seen that since the associated Poisson fibration is locally trivial, $\Ver^*$ is integrable \emph{iff} the typical Poisson fiber $(E_x,\pi_V|_{E_x})$ is integrable.

Assume now that $L$ is integrable. Then the projection $q: \G(E)\to \Pi(B)$ is a smooth surjective submersion, therefore $\ker q$ is a Lie groupoid integrating $\Ver^*$, in particular the typical Poisson fiber is integrable. Furthermore, since $\ker q= \G(\Ver^*)/\mathcal{M}$ is smooth, $\mathcal{M}$ is necessarily embedded in $\G(\Ver^*)$.

Conversely, suppose that $\mathcal{M}$ is embedded in $\G(\Ver^*)$ and consider a sequence $(\xi_n)\subset \mathcal{N}(L)$ of monodromy elements of $L$ converging to a trivial path $0_x$. Since $\ker\sharp\subset \Ver^*$, one can consider the sequence $[\xi_n]_{V}\in \G(\Ver^*)$, where $\xi_n$ is considered as a constant path. By the definition of the monodromy groups
$\mathcal{N}(L)$ controlling the integrability of $L$ (see \cite{CrFe1}), $[\xi_n]_L\in \G(L)$ is a sequence of units $[\xi_n]_L={\bf 1}_{x_n}$, therefore $[\xi_n]_{V}\in \M$. In other words, there exists a neighborhood $U$ of the identity section in $\G(L)$ such that $\NN(L)\cap U\subset \M\cap U$. Since $\mathcal{M}$ is embedded in $\G(\Ver^*)$, it follows that there exists a neighborhood $V\subset U$ of the identity section in $\G(L)$ such that $\NN(L)\cap V$ coincides with the identity section. This shows that the obstructions to integrability of $L$ vanish.
\end{proof}

\begin{ex}[Hamiltonian symplectic fibrations]
Assume $L$ that is the graph of a presymplectic form. Then $L$ identifies with $TE$ as a Lie algebroid (using the anchor map). In particular $L$ is integrable and $\G(L)$ identifies with the fundamental groupoid of $E$. Let us see how one can recover this using the previous construction.
 
In that case, $\pi_V$ is the inverse of a symplectic (vertical) form, therefore $\Ver^*$ identifies with $\Ver$ as a Lie algebroid, and $\G(\Ver^*)\simeq\G(\Ver)$, which is just a fibered version of the fundamental groupoid. Therefore the transgression map becomes $\partial:\pi_2(B)\ltimes E\to \G(\Ver)$, and as easily checked, corresponds to the usual transgression map in the homotopy long exact sequence associated to the fibration $E\to B$. It follows that $\M_x$ lies in the fundamental group $\pi_1(E_{p(x)})$ of the fiber through $x\in E$, and $\M$ is locally trivial over $E$. Now Theorem \ref{th:integrability} shows that $L$ is integrable.
  
In fact, if the fibers are compact, one can even show that the transgression map vanishes. Indeed, given a sphere in $B$, it follows from \eqref{partial:path} that the loops representing the image of the transgression map are the so-called hamiltonian loops (see \cite{MaSa}). For compact symplectic manifolds, it is a well known fact that such hamiltonian loops are always contractile.
\end{ex}

\begin{ex}[Split Poisson Structures]\label{ex:split1}
 When a coupling Dirac structure $L$ is the graph of a Poisson structure $\pi$, the decomposition \eqref{eq:Dirac:components} corresponds to a splitting $\pi=\pi_V+\pi_H$, where $\pi_H$ is a bivector field (see \cite{Vor}).
 
One may check that the corresponding connection has vanishing curvature if and only if $\pi_H$ is Poisson. The characteristic foliation of $\pi_H$ is then given by the integrable distribution $\Hor$. Moreover, it follows from the curvature identity \eqref{eq:curvature} that the connection if flat if and only if $\omega_H$ takes values in the space of Casimirs of $\pi_V$.

Let us assume that $\pi_H$ is indeed Poisson and, for the sake of simplicity, assume that $E=B\times F$ is a trivial fibration. Then one can still interpret the elements of $\M$ in terms of variations of the symplectic area of spheres. First notice that $\tilde{\gamma}^\epsilon$ is necessarily a trivial path since the connection is trivial. Furthermore, the integral in \eqref{partial:path} involves $\omega_H(\frac{\d\gamma_B}{\d t},\frac{\d\gamma_B}{\d\epsilon})$, which are Casimirs of the vertical Poisson structure on $F$. The resulting element in $\Ver^*_{x_0}$ lies in the center of the isotropy algebra at $x_0$. Thus taking the corresponding $\Ver^*$-homotopy class amounts to integrate along the $\epsilon$ variable.
\end{ex}

\begin{ex}
As a particular case of Example \ref{ex:split1} consider the trivial Poisson fibration $E=\Ss^2\times \mathfrak{so}_3^*\to \Ss^2$, where $p$ is the projection on the first factor. Let $\pi_V$ be the linear Poisson structure on the fibers $\mathfrak{so}_3^*$ of the projection and let $\Hor_{(b,x)}=T_b S^2\times\{x\}$ be the trivial connection. Then  $\omega_H$ must necessarily be of the form 
$\omega_H=f \cdot\omega$, where $\omega$ denotes the standard symplectic form on $S^2$ and $f$ is a Casimir of $\mathfrak{so}_3^*$,  i.e., a smooth function of the radius $r\in C^\infty(\mathfrak{so}_3^*)$.

One knows (see, e.g., \cite{CrFe2}) that the (usual) monodromy groups of the vertical Poisson structure at some $(b,x)\in B\times\mathfrak{so}_3^*$ are given by:
 \[\mathcal{N}(\Ver^*)_{(b,x)}= 
\begin{cases} \ 4\pi  \mathbb{Z}.\d r&\mbox{if } x\neq 0,\\ 
                         \  \{0\}   &\mbox{if } x=0,
 \end{cases}\]	
Applying the integrability criterium of Theorem \ref{th:integrability}, we see that $L$ is integrable if and only if $4\pi f'(r)$ is a rational multiple of $4\pi$, for any $r$. By a smoothness, this means that
$f'$ must be constant, thus $f(r)=\alpha r + \beta$, where $\alpha\in \mathbb{Q}$ and $\beta\in \mathbb{R}$.

One can also recover this result using the prequantization Lie algebroids of \cite{Cr} associated with a product of presymplectic spheres: on each leaf, the restricted Lie algebroid 
$L|_{S^2\times S^2\times\{v\}}$ is the prequantization of a product $(S^2\times S^2, f'(v)\omega\times \omega)$ of a presymplectic spheres. It is well known that leaf wise, 
$f'(v)$ must be a rational multiple of $\int_{S^2}\omega=4\pi$.

This example shows how rigid the integrability conditions can be: in this example, the value and the derivative of $f$ at a point entirely determines the structure.
\end{ex}




\appendix

\numberwithin{equation}{section}

\section{Actions on Lie groupoids and Lie algebroids} 
\label{sub:sec:actions}

We will have to look at various actions of Lie groups and algebras on Lie groupoids and Lie algebroids. The following diagram summarizes the various possibilities:
\[
\xymatrix{
\txt<8pc>{Lie group action on a groupoid: $\Ac:G\to\Aut(\G)$}
\ar@2{->}[r]\ar@2{->}[d]
& \txt<8pc>{Lie group action on a algebroid $\ac:G\to\Aut(A)$}
\ar@2{->}[d]\\
\txt<8pc>{Lie algebra action on a groupoid $\Aci:\gg\to\X_\text{mult}(\G)$}
\ar@2{->}[r]
&\txt<8pc>{Lie algebra action on a algebroid $\aci_*:\gg\to\Der(A)$}
}
\]
where the four corners have the following precise meaning:
\begin{itemize}
\item \textbf{Action of a Lie group $G$ on a Lie groupoid $\G$:} This means a smooth action $\Ac:G\times \G\to\G$ such that for each $g\in G$ the map $\Ac_g:\G\to\G$, $x\mapsto gx$, is
a Lie groupoid automorphism. 
\item \textbf{Action of a Lie group $G$ on a Lie algebroid $A$:} This means a smooth action $\ac:G\times A\to A$ such that for each $g\in G$ the map $\ac_g:A\to A$, $a\mapsto ax$, is
a Lie algebroid automorphism.
\item \textbf{Action of a Lie algebra $\gg$ on a Lie groupoid $\G$:} This means a Lie algebra homomorphism $\Aci:\gg\to\X_\text{mult}(\G)$, where $\X_\text{mult}(\G)\subset\X(\G)$ denotes the multiplicative vector fields in $\G$.
\item \textbf{Action of a Lie algebra $\gg$ on a Lie algebroid $A$:} This means a Lie algebra homomorphism $\aci_*:\gg\to\Der(A)$, where $\Der(A)$ is the space of derivations of the Lie algebroid $A$.
\end{itemize}
Obviously, Lie group actions on Lie groupoids and algebroids cover ordinary Lie group actions on the base manifold. Similarly, Lie algebra actions on Lie groupoids and algebroids cover ordinary Lie algebra actions on the base manifold.

\begin{rem}
We will be using the following convention: for a vector bundle $V\to M$, $\gl(V)$ will denote the space of derivations of $V$ i.e., linear maps $D:\Gamma(V)\to\Gamma(V)$ such that there exists a vector field $X_D$ (called the symbol of the derivation) satisfying:
\[ D(fs)=fDs+X_D(f)s,\quad (f\in C^\infty(M),\ s\in\Gamma(V)). \]
Derivations are the same thing as fiberwise linear vector fields on $E$. So a derivation $D$ has a flow which is a one parameter group of vector bundle automorphisms of $V$, $\varphi_t^D\in\GL(V)$.

On the other hand, for a Lie algebroid $A$, we denote by $\Der(A)$ the space of Lie algebroid derivations of $A$, i.e. those derivations $D\in\gl(A)$ satisfying:
\begin{align*} D([\al,\be])=&[D\al,\be]+[\al,D\be], \\
               \sharp(D\alpha)=&[X_D,\sharp\alpha].
\end{align*}
The corresponding flow is a one paremeter group of Lie algebroid automorphisms of $A$: $\varphi_t^D\in\Aut(A)$.
\end{rem}

The arrows in the diagram above represent natural constructions:
\begin{itemize}
\item From a Lie group action on a groupoid $\Ac:G\to\Aut(\G)$ we can build a Lie group action
$\ac:G\to\Aut(A)$, where $A=A(\G)$: since each $\Ac_g:\G\to\G$ is a Lie groupoid automorphism
it induces a Lie algebroid automorphism $\ac_g:A\to A$:
\[\ac_g:=\d\Ac_g|_A, \]  here we identified $A_m$ with the tangent space to the $\s-$fibers at the identity $1_m$). This defines the action at the level of the Lie algebroid. The corresponding action on sections of $A$ will be denoted  $\ac_*$, that is \[(\ac_g)_*(\beta)_m:=\ac\circ\beta(g^{-1}m),\  (\beta\in \Gamma(A)).\]
\item From a Lie group action on a algebroid $\ac_*:G\to\Aut(A)$, we can build a Lie algebra action $\aci_*:\gg\to\Der(A)$ by differentiation:
\[ (\aci_\xi)_*(\beta):=\left.\frac{\d}{\d t}(\ac_{\exp(-t\xi)})_*(\beta)\right|_{t=0}\quad (\xi\in\gg, \beta\in\Gamma(A)). \]
The fact that the action is by Lie algebroid automorphisms, implies that $(\aci_\xi)_*$ is a Lie algebroid derivation.
 Moreover, the corresponding vector field on $A$ is given by:
\[ (\aci_\xi)_a:=\left.\frac{\d}{\d t}\ac_{\exp(-t\xi)}(a) \right|_{t=0}\quad (\xi\in\gg, a\in A).\]
\item From a Lie group action on a groupoid $\Ac:G\to\Aut(\G)$ we can also build a Lie algebra action $\Aci:\gg\to\X_\text{mult}(\G)$ by differentiation:
\[ \Aci(\xi)_x:=\left.\frac{\d}{\d t}\Ac_{\exp(-t\xi)}(x)\right|_{t=0}\quad (\xi\in\gg, x\in\G). \]
The fact that the action is by Lie groupoid automorphisms, implies that $\Aci(\xi)$ is a multiplicative vector field in $\G$.

\item Finally, from a Lie algebra action on a groupoid $\Aci:\gg\to\X_\text{mult}(\G)$ we obtain
a Lie algebra action $\aci_*:\gg\to\Der(A)$, where $A=A(\G)$: if we identify sections of $A$
with right invariant vector fiels in $\G$, then we can define $\aci$ by:
\[ (\aci_\xi)_*:=[\Aci(\xi),-]_A\quad (\xi\in\gg).\]
This is well-defined since the Lie bracket between a multiplicative and a right invariant vector field is a right invariant vector field.
\end{itemize}
With these precisions, the commutativity of the diagram becomes obvious.
\begin{rem}
A multiplicative vector field $X\in\X(\G)$ can be thought of
as a Lie groupoid homomorphism $\G\to T\G$. Its base map is just an ordinary vector field
$M\to TM$. In the last construction, for each $\xi\in\gg$, the symbol of the derivation $(\aci_\xi)_*$ is precisely the base map of the multiplicative vector field $\Aci_\xi$.
\end{rem}

Under apropriate assumptions one can also invert the arrows in the diagram above, namely:
\begin{itemize}
\item One can invert the horizontal arrows (integrate actions on Lie algebroids
to actions on Lie groupoids) if $\G=\G(A)$, the source 1-connected Lie groupoid 
integrating $A$.
\item One can invert the vertical arrows (integrate Lie algebra actions 
to Lie group actions) if $G=G(\gg)$, the source 1-connected Lie group 
integrating $\gg$, and if the infinitesimal actions are complete (the flows are 
defined for all $t\in\Rr$).
\end{itemize}
The reader should be able to fill in the details.

\subsection{Inner actions} 
\label{sub:sec:inner:actions}

For Lie group actions as introduced in the previous paragraph, there is a notion of \emph{inner actions} as we explain below. First recall the notion of bisection for a Lie groupoid $\G\tto M$. 
\begin{defn}
A bisection $b:M\to \gg$ is a smooth section of the source map such that $\t\circ b$ is a diffeomorphism of $M$. 
\end{defn}
The space $\text{Bis}(\G) $ of bisections has natural structure of a group, where the unit, mutliplication, inverses are given as follows: $1(x)={\bf 1}_x$, 
 $(b.b')(x)=b\circ\t\circ b'(x)\cdot b'(x)$, and
 $b^{-1}(x)={\bf i}\circ b\circ (\t\circ b)^{-1}(x)$. Here $b,b'\in \text{Bis}(\G)$, $x\in M$, ${\bf 1}: M\to \G$ denotes the unit map of $\G$, and ${\bf i}:\G\to \G$ the inverse map. Clearly, $\text{Bis}(\G)\to \text{Diff}(M),\ b\mapsto \t\circ b$ is a morphism of groups.

The notion of inner action for Lie groupoids follows immediately from the following definitions:
\begin{itemize}
\item An \textbf{inner Lie groupoid automorphism} is Lie groupoid automorphism $\Phi:\G\to\G$ of the form:
\[ \Phi(x)=b(\t(x))\cdot x\cdot b(\s(x))^{-1}. \]
for some bisection $b:M\to\G$. They clearly form a subgroup $\InnAut(\G)\subset\Aut(\G)$.
\item A \textbf{inner Lie algebroid automorphism} is a Lie algebroid automorphism $\phi:A\to A$ of the form:
\[ \phi=\varphi_{1,0}^{D_\al},\]
for some time dependent section $\al_t\in\Gamma(A)$. Here $t\mapsto \varphi_{t,0}^{D_\al}$ denote the flow of the time dependent derivation $D_{\al_t}:=[\al_t,-]$).
\comment{oli: I made $\alpha$  time-dependent here} They generate a subgroup $\InnAut(A)\subset\Aut(A)$.
\item A \textbf{mutiplicative exact vector field} is a multiplicative 
vector field $X\in\X_\text{mult}(\G)$ of the form:
\[ X=\ri{\al}-\le{\al},\]
where $\al$ is a section of $A=A(\G)$, and $\ri{\al}$ and $\le{\al}$ are the right and left invariant vector fields in $\G$ determined by $\al$. They form a Lie subalgebra $\X_\text{exact}(\G)\subset \X_\text{mult}(\G)$.
\item An \textbf{inner derivation} is a Lie algebroid derivation $D\in\Der(A)$ of the form:
\[ D=[\al,-]_A,\]
for some section $\al\in\Gamma(A)$. They clearly form a Lie subalgebra $\InnDer(A)\subset\Der(A)$.
\end{itemize}
Now, one can define inner actions in a more or less obvious fashion. We obtain a diagram as above:
\[
\xymatrix{
\txt<8pc>{inner Lie group action on a groupoid: $\Ac:G\to\InnAut(\G)$}
\ar@2{->}[r]\ar@2{->}[d]
& \txt<8pc>{inner Lie group action on a algebroid $\ac:G\to\InnAut(A)$}
\ar@2{->}[d]\\
\txt<8pc>{inner Lie algebra action on a groupoid $\Aci:\gg\to\X_\text{exact}(\G)$}
\ar@2{->}[r]
&\txt<8pc>{inner Lie algebra action on a algebroid $\aci_*:\gg\to\InnDer(A)$}
}
\]
In this work, we will mainly consider inner actions associated with a Lie groupoid morphism $\Psi:G\times M\to\G$ by the formula:
\begin{equation} 
\label{eq:inner:group:action}
\Ac_g(x)=\Psi(g,\t(x))\cdot x\cdot \Psi(g,\s(x))^{-1}\quad (g\in G,x\in \G).
\end{equation}
Notice that the map $\Psi$ covers the ordinary action $G\times M\to M$ on the base. Furthermore, one may check that $\Psi:G\times M\to\G$ is a Lie groupoid morphism \emph{iff} the map $G\to \text{Bis}(\G)$,\ $g\mapsto b^g(x):=\Psi(g, x)$ is a group morphism covering the usual Lie group action of $G$ on $M$.

Similarly, the inner Lie algebra actions on a Lie algebroid $\aci_*:\gg\to\InnDer(A)$ will come associated with a Lie algebroid morphism $\psi:\gg\times M\to A$ (covering the identity on $M$) such that:
\begin{equation} 
\label{eq:inner:algebra:action}
(\aci_\xi)_*=[\psi_*(\xi),-]_A,\quad (\xi\in\gg)
\end{equation}
where $\psi_*(\xi)\in\Gamma(A)$ is defined by $\psi_*(\xi)_m=\psi(\xi,m)$, for any $m\in M$.
 The map $\psi_*:\gg\to \Gamma(A)$ covers the ordinary
Lie algebra action $\gg\to\X(M)$ on the base. Moreover, $\psi_*$ is a Lie algebra morphism covering the infinitesimal action $\gg\to\X(M)$ \emph{iff} $\psi$ is a Lie algebroid.

\begin{rem}
One should note, however, that maps $\Psi$ and $\psi$ as above do not 
need to define morphisms from the action groupoid and algebroid into $\G$ and $A$ in order for \eqref{eq:inner:group:action} and \eqref{eq:inner:algebra:action} to induce inner actions. In this paper though, we will always assume that it is the case.
\end{rem}

One may now easily check that:

\begin{prop}
Let $G\times M\to M$ be an action of a Lie group on a manifold. Then any
homorphism $\Psi:G\ltimes M\to\G$ from the action Lie groupoid to a Lie groupoid $\G$ 
determines by the formula \eqref{eq:inner:group:action} a inner action of $G$ on $\G$ that
covers the action on $M$.

Similarly, let $\g\to\X(M)$ be an action of a Lie algebra on a manifold. Then any
homorphism $\psi_*:\gg\ltimes M\to A$ from the action Lie algebroid to a Lie algebroid $A$ 
determines by the formula \eqref{eq:inner:algebra:action} an inner action of $\gg$ on $A$ that
covers the infinitesimal action on $M$.
\end{prop}

The relevant notion for this work is the following:

\begin{defn}\label{prehamiltonian}
 A pre-hamiltonian action of a Lie group $G$ on a Lie algebroid $A$ with \textbf{pre-hamiltonian moment map} $\psi_*:\g\to \Gamma(A)$ is an action $\ac:G\to \Aut(A)$ such that:
\begin{itemize}
 \item $\left.\frac{\d}{\d t}(\ac_{\exp(-t\xi)})_*(\beta)\right|_{t=0}=[\psi_*(\xi),\beta]_A\quad (\xi\in\gg,\ \beta\in\Gamma(A)),$
 \item $\psi_*$ is a $G$-equivariant morphism of Lie algebras.
\end{itemize}
\end{defn}
Note that the $G$-equivariance condition is automatically satisfied whenever $G$ is connected.

\subsection{Integration of inner actions}

Let us now see in which circumstances one is able to invert arrows in the last diagram.

\begin{prop}\label{prop:int:action:ss}
Let $\gg\to\X(M)$ be a complete Lie algebra action and $\psi:\gg\ltimes M\to A$ a Lie algebroid morphism from the action Lie algebroid to a Lie algebroid $A$. For any Lie groupoid $\G$ integrating $A$, the associated inner action $\aci_*:\gg\to\InnDer(A)$ integrates to a
inner action $\Ac:G(\gg)\to\InnAut(\G)$, where $G(\gg)$ is the 1-connected Lie group integrating $\gg$.
\end{prop}

\begin{proof}
By the assumptions, we have a Lie group action $G(\gg)\times M\to M$ and the corresponding action groupoid $G(\gg)\ltimes M\tto M$ is source 1-connected.
Since every Lie algebroid morphism $\psi:\gg\ltimes M\to A$ integrates to 
a Lie groupoid morphism $\widetilde{\Psi}:G(\gg)\ltimes M\to\G(A)$. Denote by $\Psi$ the composition of $\widetilde{\Psi}$ with the natural projection $\G(A)\to \G$, then one obtains an inner action $\Ac:G(\gg)\to\InnAut(\G)$ by the formula \eqref{eq:inner:group:action}. As easily checked, it integrates the 
inner action $\aci_*:\gg\to\InnDer(A)$.
\end{proof}

\begin{rem}
Notice that the above result is slightly better than the integration of non inner actions we refered to in the end of the preceeding subsection: in general, in order to integrate a Lie algebra action of $\gg$ on a Lie algebroid $A$ to a Lie group action of $G$ on a Lie groupoid $\G$, we need \emph{both} $G$ to be 1-connected and $\G$ to be source 1-connected (and the action to be complete).

This is important for our purposes, as $G$ is the structure group of a principal bundle, thus its topology imposed. So we need to refine the Prop. \ref{prop:int:action:ss} to non simply connected, and non connected groups.
\comment{oli: not 'later on'}
\end{rem}

 Assume now that we want to integrate the inner action $\aci:\gg\to\InnDer(A)$ associated with $\psi:\gg\ltimes M\to A$ to an action of a connected (but not necessarily 1-connected) Lie group $G$ with Lie algebra $\gg$. Recall that any connected Lie group $G$ and its fundamental group fit into an exact sequence of Lie groupoids as follows:
\begin{equation}\label{ex:groups}
\SelectTips{cm}{}\xymatrix@C=15pt{ 1\ar[r]&\pi_1(G)\ar[r]& G(\gg)\ar[r]& G\ar[r]& 1.} 
\end{equation}
Then it is easy to prove the following:
\begin{prop}
Let $G$ be a connected Lie group acting on a manifold $M$, $\G$ a Lie groupoid over $M$, and a Lie groupoid morphism $\overline{\Psi}:G(\gg)\ltimes M\to\G$.

Then $\overline{\Psi}$ descends to a Lie groupoid morphism $\Psi:G\ltimes M\to\G$ if and only if it  
takes values in units when restricted to $\pi_1(G)\ltimes M\subset G(\gg)\ltimes M$, namely:
\[\overline{\Psi}(\pi_1(G)\ltimes M)={\bf 1}_M.\]
\end{prop}

\begin{rem}
Although $\psi:\gg\ltimes M\to A$ might does not integrate to $\Psi:G\ltimes M\to\G$, it may still integrate to a morphism $\Psi':G\ltimes M\to\G'$ for a smaller Lie groupoid $\G'$ integrating $A$. This can be decided as follows.

 First we integrate $\psi:\gg\ltimes M\to A$ to a Lie groupoid morphism $\widetilde{\Psi}:G(\gg)\ltimes M\to\G(A)$ with values in the source 1-connected Lie groupoid 
integrating $A$. Then, given any connected Lie group $G$ integrating $\gg$, we introduce 
a bundle of groups $\Delta$ over $M$, defined in the following way:
\[ 
\Delta:=\widetilde{\Psi}(\pi_1(G)\ltimes M)\subset \G(A)
\]
Recall that if $\Delta$ is a totally disconnected wide \textbf{normal subgroupoid} of $\G(A)$, \emph{i.e.} an embedded Lie subgroupoid of $\G(A)$ which is a bundle of Lie groups and such that $a.u.a^{-1}\in N_{\t(a)}$ for any $a\in \G(A)$ and $u\in N_{\s(a)}$, then the quotient $\G(A)/\Delta$ is a Lie groupoid (see the discussion in \cite[Thm. 1.4]{Gu}). Furthermore, it is easy to see that $\G(A)/\Delta$ integrates $A$. Therefore, we obtain a Lie groupoid morphism:
\[
\Psi:G\ltimes M\to\G(A)/\Delta.
\]
Of course, $\psi$ integrates to a morphism $G\ltimes M\to\G$ whenever $\G$ is covered by $\G_0:=\G(A)/\Delta$.
\end{rem}

The following theorem, together with the discussion above, justifies the definition of a pre-hamiltonian action. Notice that no assumption is made about simply connectedness, nor connectedness of $G$.


\begin{thm}\label{int:ham:action}
Let ${\bf{a}}:G\rightarrow \Aut(A)$ be a pre-hamiltonian action of a Lie group $G$ on a Lie algebroid $A$ with pre-moment map $\psi_*:\gg\to \Gamma(A)$ and:
\begin{itemize}
\item $\psi: \g\ltimes M \to A$ the Lie algebroid associated with $\psi_*$,
\item $\widetilde{\Psi}:G(\gg)\ltimes M\to \G(A)$ the groupoid morphism integrating  $\psi$, 
\item $\Delta\subset \G(A)$, the subset defined by:  $\Delta:=\widetilde{\Psi}(\pi_1(G)\ltimes M)\subset \G(A)$.
\end{itemize}
Then the following assertions hold:
\begin{enumerate}[i)]
\item $\Delta $ is a wide, normal, totally disconnected subgroupoid of $\G(A)$,
\item ${\bf a}$ integrates to a groupoid action $\Ac:G\rightarrow \Aut(\G(A)/\Delta)$,
\end{enumerate}
 Moreover, $\G(A)/\Delta$ is a Lie groupoid \emph{iff} $\Delta\subset\G(A)$ is an embedding, in which case $\widetilde{\Psi}$ descends to a Lie groupoid morphism $\Psi:G\times M\to \G(A)/\Delta$ such that the following relations hold:
\begin{align}
                 \Ac_h(x)\quad     &=\Psi_n(h)\cdot x \cdot\Psi_m(h)^{-1}\label{eq:int:ham:action1} \\
                \Psi_{gm}(ghg^{-1})&=\quad \Ac_g(\Psi_m(h))\label{eq:int:ham:action2}
\end{align}
for any $h\in G^\circ,\ x\in\G(A)/\Delta,\ g\in G$, where $m:=\s(x)$ and  $n:=\t(x)$.
\end{thm}


\begin{proof}
Since $G$ acts on $A$ by Lie algebroid automorphisms, one can lift $\ac$ into an action $\widetilde{\Ac}$ of $G$ on $\G(A)$ by groupoid automorphisms as follows:
\[\widetilde{\Ac}_g([q]_{A}):=[\ac_g\!\circ q]_{A},\quad (q\in P(A),\ g\in G).\]
Consider now an element $h$ lying in the neutral component $G^\circ$ of G, and an $A$-path $q$. We extend  $q$ into a time dependant section of $A$ (that we still denote by $q$) and consider any $\g$-path $\xi:\epsilon\mapsto \xi_\epsilon\in\g$ that  induces a path $h_\epsilon$ in $G^\circ$ between the identity and $h$. By the construction,  $(\ac_{h_\epsilon})_*(q_t)$ is a solution of the following evolution equation:
\[ \bigl[(\ac_{h_\epsilon})_*(q_t),\psi_*(\xi_\epsilon)\bigr]_A=\frac{\d}{\d\epsilon} (\ac_{h_\epsilon})_*(q_t)\quad(\epsilon,t\in I).\]
Then by the  Prop. \ref{hgeom} with $\frac{\d}{\d t}\xi=0$, we obtain:
\begin{equation}\label{eq:ac:conj1} \widetilde{\Ac}_h([q]_{A})=\widetilde{\Psi}_y([\xi]_\g)\cdot[q]_{A}\cdot\widetilde{\Psi}_x([\xi]_\g)^{-1}.
\end{equation}
In particular, if $h=1$ that is, $\xi_\epsilon$ induces a loop in $G^\circ$, we obtain that $\Delta$ is a normal subgroupoid of $\G(A)$.

Next, we have to make sure that $\widetilde{\Ac}$ induces an action on $\G(A)/\Delta$, so we need to check that $\tilde{A}_h(\Delta)=\Delta$. For this, we apply the  equation \eqref{eq:ac:conj1} with $q=\widetilde{\Psi}{[\eta]_\g}$, where $\eta$ is a $\g$-path inducing a loop in $G$, and we use successively the fact that $\widetilde{\Psi}$ is a Lie groupoid homomorphism, then that $\pi_1(G)$ lies in $G(\g)$ as a normal subgroup. The first relation then easily follows. The second one is obtained by using the equivariance condition in the Def. \ref{prehamiltonian}.
\end{proof}
Note that when the context makes no confusion, we may use notations loosely, using for instance $\ac_\xi$, $\psi$ and so on, instead of $(\ac_\xi)_*$, $\psi_*$ and so on.
\begin{ex}\label{ex:non-smooth-quotient}
Here a basic example where the resulting groupoid $\G(A)/\Delta$ is not smooth. Consider the $1$-dimensional (abelian) Lie algebra $\z=\Rr$, its dual $\z^*$ endowed with the trivial linear Poisson structure, and $A:=T^*\z^*\simeq \z\ltimes \z*$ the corresponding Lie algebroid (with trivial bracket and anchor). Then $A$ integrates to a bundle of Lie groups $Z\ltimes\z^*$, where $Z$ identifies with the Lie group $\Rr$.

 Consider furthermore the trivial action of $G:=S^1\simeq \mathbb{R}/2\pi\mathbb{Z}$ on $\z^*$. Then any application $J:\z^*\to \text{Lie}(S^1)$ can be chosen to be a moment map. From the Lie algebroid point of view (see the def. \ref{prehamiltonian}) we only have a pre-hamiltonian moment map $\psi_*:\text{Lie}(S^1)\to\Gamma(\z\ltimes\z^*),$ $X\mapsto \d J(X)$. The corresponding Lie algebroid morphism $\psi:\text{Lie}(S^1)\ltimes\z^*\to\z\ltimes\z^*$ is given by $(X,z)\mapsto(\d J_z(X),z)$.  It integrates to a Lie groupoid morphism $\mathbb{R}\ltimes\z^*\to Z\ltimes\z^*$ given by $(\theta,z)\mapsto(\d J_z(\theta),z)$. The exact sequence \eqref{ex:groups} reads:
\[\SelectTips{cm}{}\xymatrix@C=15pt{ 1\ar[r]&2\pi\mathbb{Z}\ar[r]& \mathbb{R}\ar[r]& S^1\ar[r]& 1,}\]
hence we obtain $\Delta_z=\left\{(\d J_z(2k\pi),z),\ k\in\Zz\right\}$. Clearly, $\Delta$ defines a normal subgroupoid of $Z\ltimes\z^*$, however, $Z\ltimes\z^*{/\Delta}$ is not smooth if $\d J$ vanishes at some point $z_0$.
\end{ex}

\bibliographystyle{amsalpha}


\end{document}